\documentclass[sn-mathphys,Numbered]{sn-jnl}

\usepackage{graphicx}%
\usepackage{multirow}%
\usepackage{amsmath,amssymb,amsfonts}%
\usepackage{amsthm}%
\usepackage{mathrsfs}%
\usepackage[title]{appendix}%
\usepackage{xcolor}%
\usepackage{textcomp}%
\usepackage{manyfoot}%
\usepackage{booktabs}%
\usepackage{algorithm}%
\usepackage{algorithmicx}%
\usepackage{algpseudocode}%
\usepackage{listings}%

\usepackage[T1]{fontenc}

\usepackage{mwe} 
\usepackage{caption}
\usepackage{amsfonts,mathrsfs}
\usepackage{amsmath, amsthm, amssymb}

\usepackage{float}
\usepackage{subfigure}
\newfloat{figure}{H}{lof}
\newfloat{table}{H}{lot}
\floatname{figure}{\figurename}
\floatname{table}{\tablename}

\usepackage{mathtools}
\usepackage{etoolbox}

\usepackage{enumerate}

\usepackage{xcolor}
\usepackage{tcolorbox}
\usepackage{colortbl}
\usepackage{bbm}
\usepackage{nicematrix}

\usepackage[latin1]{inputenc}

\usepackage{float}
\newfloat{figure}{H}{lof}
\newfloat{table}{H}{lot}
\floatname{figure}{\figurename}
\floatname{table}{\tablename}

\numberwithin{equation}{section}

\theoremstyle{thmstyleone}
\newtheorem{theorem}{Theorem}[section]

\newtheorem{assumption}[theorem]{Assumption}

\newtheorem{corollary}[theorem]{Corollary}

\newtheorem{definition}[theorem]{Definition}

\newtheorem{lemma}[theorem]{Lemma}

\newtheorem{proposition}[theorem]{Proposition}
\newtheorem{remark}[theorem]{Remark}

\numberwithin{equation}{section}

\newcommand{\dx}{\mathrm{d} x}

\newcommand{\dt}{\mathrm{d} t}

\renewcommand{\d}{\mathrm{d}}

\newcommand{\R}{\mathbb{R}}
\newcommand{\C}{\mathbb{C}}

\newcommand{\Z}{\mathbb{Z}}

\newcommand{\Li}{\mathrm{L}}
\renewcommand{\L}{\mathcal{L}}

\renewcommand{\d}{{\rm \, d}}
\renewcommand{\Re}{{\rm Re \, }}
\newcommand{\Ra}{{\rm R \, }}
\newcommand{\Nu}{{\rm N\, }}

\begin{document}

\title[\textbf{A return-to-home model with commuting people and workers}]{\textbf{A return-to-home model with commuting people and workers}}


\author{Pierre Magal}

%

\affil{Univ. Bordeaux, IMB, UMR 5251, F-33400 Talence, France}

%


\abstract{	This article proposes a new model to describe human intra-city mobility. The goal is to combine the convection-diffusion equation to describe commuting people's movement and the density of individuals at home. We propose a new model extending our previous work with a compartment of office workers. To understand such a model, we use semi-group theory and obtain a convergence result of the solutions to an equilibrium distribution. We conclude this article by presenting some numerical simulations of the model.  }




\maketitle

\section{Introduction}
Understanding human intra-city displacement is crucial since it influences populations' dynamics. Human mobility is essential to understand and quantifying social behavior changes.  In light of the recent COVID-19 epidemic outbreak, human travel is critical to know how a virus spreads at the scale of a city, a country, and the scale of the earth, see \cite{Ruan, Ruan2017}. 

\medskip 
We can classify human movement into: 1) short-distance movement: working, shopping, and other intra-city activities; 2) long-distance movement: intercity travels, planes, trains, cars, etc. These considerations have been developed recently in \cite{BHG,GHB, KSZ, MS}. A global description of the human movement has been proposed (by extending the idea of the Brownian motion) by considering the Lévy flight process. The long-distance movement can also be covered using patch models (see Cosner et al.	\cite{CBCI} for more results). 

\medskip 
The spatial motion of populations is sometimes modeled using Brownian motion and diffusion equations. For instance, reaction-diffusion equations are widely used to model the spatial invasion of populations both in ecology and epidemiology. We refer, for example, to Cantrell and Cosner \cite{Cantrell}, Cantrell, Cosner, and Ruan \cite{Cantrell-Cosner-Ruan}, Murray \cite{Murray}, Perthame \cite{Perthame}, Roques \cite{Roques} and the references therein. In particular, the spatial propagation for the solutions of reaction-diffusion equations has been observed and studied in the 30s by Fisher \cite{F} and Kolmogorov, Petrovski, and Piskunov \cite{KPP}. Diffusion is a good representation of the process of invasion or colonization for humans and animals. Nevertheless, once the population is established, the return-to-home process (i.e., diffusion-convection combined with return-to-home) seems to be more suitable for describing the movement of human daily life.

\medskip 
A good model for intra-city mobility should also incorporate population density in the city. Figure \ref{Fig1} represents the evolution of the population density in Tokyo. This type of problem have been consider by geographer long ago, and we refer to the book of  \cite{Pumain} for nice overview on this topic.

\begin{figure}[H]
	\begin{center}
		\hspace{-0.9cm}	\includegraphics[scale=0.5]{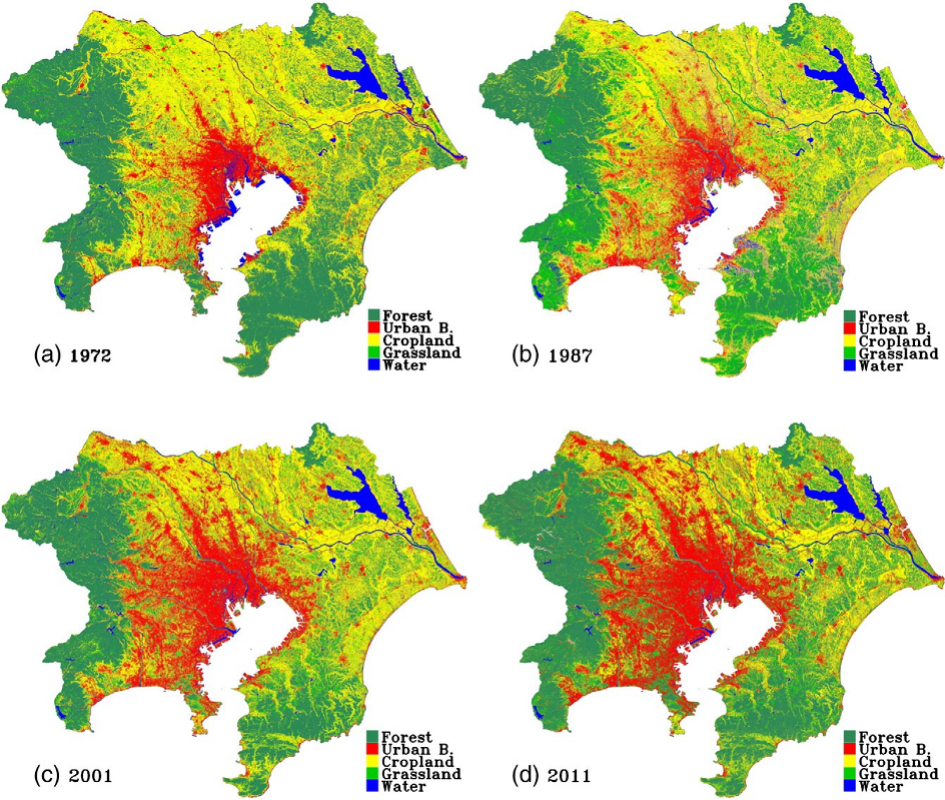} \hspace{1cm}
	\end{center}
	\caption{\textit{The above figure represents the evolution of the density of individuals (at home) in Tokyo city. This figure is taken from \cite{BY}.}}\label{Fig1}
\end{figure}

\medskip 
Ducrot and Magal \cite{DM-2022} previously proposed a model with return-to-home with two classes of people, the travelers and the people at home. The present article aims to improve this previous model by introducing a third compartment composed of immobile individuals composed mostly of office workers.

\medskip 
In other words, we are trying to model commuting people in a city. This process combines several aspects; some are summarized in Figure \ref{Fig2}. In \cite{CPB}, a patches model was proposed to describe commuting people. To our best knowledge, our approach using partial differential equations is new, and we believe that such an approach is very robust. Here, we model the tendency of commuters to travel in a city, and the diffusion takes care of the uncertainty around a tendency (which is modeled by a transport term). For instance, people going to work may change sometime their travel (to buy something, for example).

\begin{figure}
	\begin{center}
		\includegraphics[scale=1]{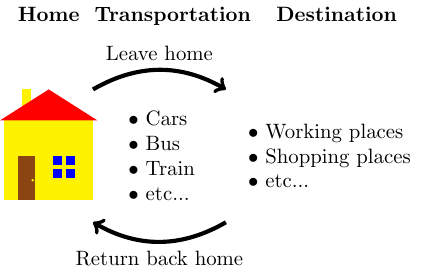}
	\end{center}
	\caption{\textit{Principle of the return-home model.}}\label{Fig2}
\end{figure}

\medskip 
The plane of the paper is the following. In section \ref{Section2}, we present the model. Section \ref{Section3} focuses on the motion of travelers by using a linear diffusion-convection equation in $L^1\left(\R^2\right)$. In section \ref{Section3}, we present an $L^1$ semigroup theory and prove the positivity and the preservation of the total mass of individuals. In section \ref{Section4}, we investigate the asymptotic behavior of the return-to-home model. Section \ref{Section5} presents a  hybrid model where the home locations are discrete. Section \ref{Section6} presents some numerical simulations of a hybrid model on a $\Omega=[0,1] \times [0,1]$. In section \ref{Section7}, we conclude the paper by discussing some perspectives. The appendix section \ref{SectionA} is devoted to the model on a bounded domain and its numerical scheme. 

\section{Eulerian formulation of the model}	

\label{Section2}
The principle of the model is described in Figure \ref{Fig4}. After leaving home, people spend some time commuting to their working places, and after spending some time at work, they return home. In the model, the average time spent at home will be $1/ \gamma$, the average time spent commuting is $1/ \alpha$, the average time spent at work in $1/ \chi$.

\medskip 
The average time spent at home $1/ \gamma$ should be approximately  equal to $12$ hours ($=0.5$ day) and the time spent at works  $1/ \chi$   should be approximately  equal to $10$ hours ($=0.41$ day)  will be much longer than the average time spent commuting $1/ \alpha$ which should be approximately  equal to $2$ hours ($=0.08$ day). But the point is to get a "simple" model to describe the movement of people using diffusion and convection. 

\medskip 
The parameters $1/ \gamma$,  $1/ \chi$, and $1/ \alpha$ may change with time, for example, during the lockdown due to an epidemic outbreak.

\begin{figure}
	\begin{center}
		\includegraphics[scale=0.5]{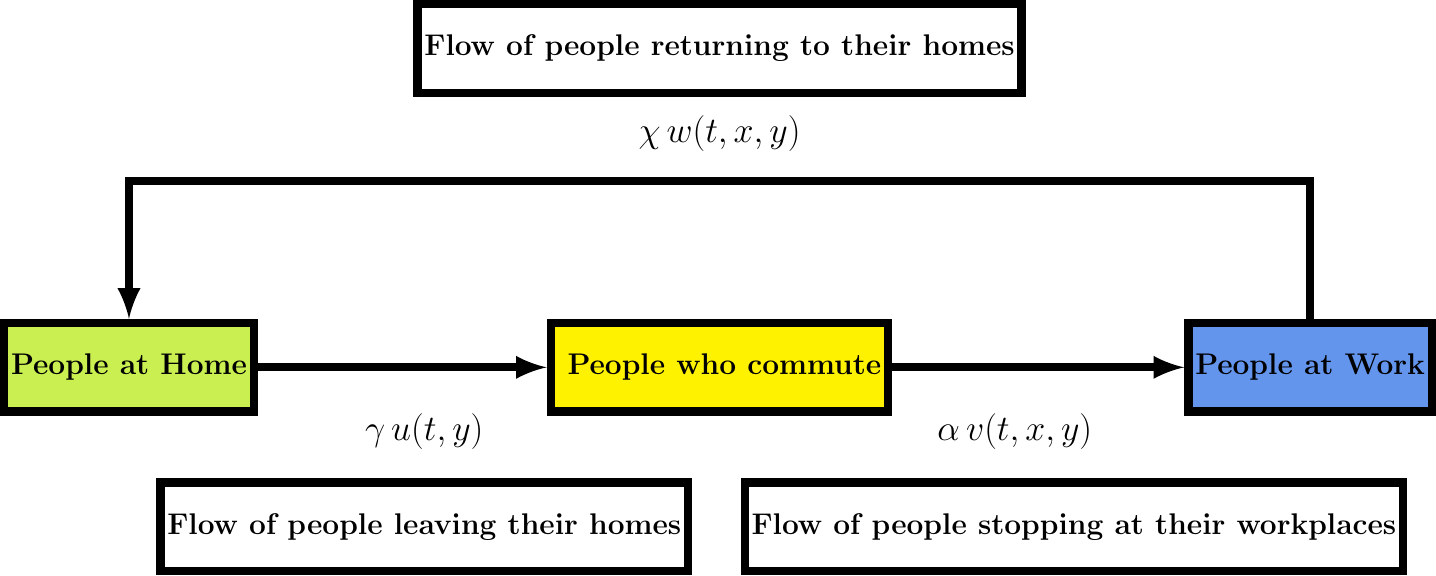}
	\end{center}
	\caption{\textit{Flowchart.}}\label{Fig3}
\end{figure}
Here, for simplicity, we focus on people who leave their homes to go to work. Therefore, the model is not focusing on people leaving their homes and spending a little time shopping, practicing their hobbies, etc... We consider this in the model by considering some random fluctuation around the main activity, which is working. Another simplification in the model is that people at work no longer move. So here we look at people working in offices or factories, and we neglect the people moving within the city for their job (ex., taxi drivers, etc...). So the model intends to capture only a part of the workers' movement.

We define the distribution of population   $y \in \R^2 \mapsto u(t,y) \in \R$ is the distribution of the population of the people staying at home at time $t$. That is to say that, for any subdomain $\omega \subset \R^2,$ 
$$
\int_{\omega   } u(t,y) dy \in \R,
$$
is the number of people staying at home with their home located in $\omega$ at time $t$. 

\medskip 
Let $y \in \R^2$ be the home location individuals. Then the distribution $x \to v(t,x,y)$ is the distribution of travelers who are going to their working place, some  shopping place, etc... and which are coming from a  home located at the position $y$.  That is to say that, for any subdomain $\omega \subset \R^2$ 
$$
\int_{\omega   } v(t,x,y) dx,
$$
is the number of \textit{travelers}  located in the region $\omega$ at time $t$ coming  from a home located at the position $y$.

\medskip 
The distribution $x \to w(t,x,y)$ is the distribution of individuals who arrived at their destination. Those people stay for a random time at their working place, a shopping place, and others before returning home. The home location of the distribution $x \to w(t,x,y)$  is $y$. We assume for simplicity that those people are no longer moving. That is, for any subdomain $\omega \subset \R^2$ 
$$
\int_{\omega   } w(t,x,y) dx,
$$
is the number of people who arrived at their destination in the subdomain $\omega$ at time $t$ and are not yet back home.

\medskip 
To simplify the analysis of the model, we consider the home location $y$ as a parameters of the model, and we use the notations
$$
u_y(t)=u(t,y), \, v_y(x)=v(x,y), \text{ and } w_y(x)=w(x,y).
$$ 
The return home model is the following  for each $y \in \R^2$, the system 
\begin{equation}\label{2.1}
	\left\{\begin{array}{ll}
		\partial_{t} u_y(t)= \chi \int_{\R^2}w_y(t,x)\dx-\gamma u_y(t),   \vspace{0.2cm} \\
		\partial_{t} v_y(t,x)=\varepsilon^2\Delta_{x}v_y(t,x)- \mathbf{\nabla}_x \cdot  \left( v_y\, \mathbf{C_y}	 \right)-\alpha v_y+\gamma g(x-y)u_y(t),    \vspace{0.2cm} \\
		\partial_{t} w_y(t,x)=\alpha v_y(t,x)-\chi w_y(t,x), 
	\end{array}\right.
\end{equation}
with the initial distribution 
\begin{equation} \label{2.2}
	\left\{ 	\begin{array}{l}
		u_y(0)=u_{y0} \in \R_+, \vspace{0.2cm} \\
		v_y(0,x)=v_{y0}(x) \in L^1_+ \left(\R^2\right), \vspace{0.2cm} \\
		\text{ and } \vspace{0.2cm} \\
		w_y(0,x)=w_{y0}(x) \in L^1_+ \left(\R^2\right).
	\end{array}	
	\right. 
\end{equation}
\begin{remark} \label{RE2.1}
	We refer to \cite{DM-2022} for an approach allowing the integrability of $u(t,x,y)$ with respect to both $x$ and $y$ for each $t>0$.  
\end{remark}
\begin{remark} \label{RE2.2} Through the paper for the Banach, we use $\Li^1 \left(\R^2\right)$ instead of $\Li^1 \left(\R^2, \R \right)$ to simplify the notations.  We will only specify the rang of maps whenever it is not equal to $\R$. 
\end{remark}

In the model, the map $x\to g(x-y)$ is a Gaussian distribution representing the location of a house centered at the position $y\in\R^2$. The function $g$ is defined by 
\begin{equation}\label{2.3}
	g(x_1,x_2)=\dfrac{1}{2\pi\sigma^2}e^{-\dfrac{x_{1}^{2}+x_{2}^{2}}{2\sigma^2}}.
\end{equation}
That is a Gaussian distribution centered at $0$, and with standard deviation  $\sigma>0$. Note that for all $y\in\R^2$, the translated map $g(\cdot-y)$ satisfies
$$
\int_{\R^2}g(x-y)\dx=1 \text{ and }\int_{\R^2}x g(x-y)\dx=y.
$$
In the model, $\Delta_{x}v_y $ is the Laplace operator of x with respect to the variable $x = (x_1, x_2)\in\R^2$. That is,
$$
\Delta_{x}v_y(t,x)=\partial^{2}_{x_1} v_y(t,x)+\partial^{2}_{x_2} v_y(t,x).
$$
The operator $\mathbf{\nabla}_x\cdot\left(v_y \, \mathbf{C_y}	\right) $ is the divergence of $v_y  \, \mathbf{C_y}	$  with respect to the variable $x = (x_1, x_2)\in\R^2$. That is,
$$
\mathbf{\nabla}_x \cdot\left(v_y(t,x) \, \mathbf{C_y}	(x)\right) =\partial_{x_1}  \biggl( v_y(t,x) \, \mathbf{C_y}(x)_1\biggr) +\partial_{x_2}  \biggl( v_y(t,x) \, \mathbf{C_y}(x)_2\biggr), 
$$
where 
$$ 
\mathbf{C_y}	(x)
=\left(\begin{array}{c}
	\mathbf{C_y}(x)_1\\ 
	\mathbf{C_y}(x)_2
\end{array}\right) \in \R^2,
$$ 
is the speed of individuals located at the position $x  \in \R^2$ and coming from a home located at the position $y  \in \R^2$.

\medskip 
The density of individuals per house remains constant with time. That is 
\begin{equation}\label{2.4}
	n(y)=	u_y(t)+ \int_{\R^2}v_y(t,x)\dx+  \int_{\R^2}w_y(t,x)\dx, \forall t \geq 0, \forall y \in \R^2, 	
\end{equation}
where $n(y)$ is the density of home in $\R^2$. That is, for each subdomain $\omega \subset \R^2$ 
$$
\int_{\omega   } n(y) dy,
$$
is the number of people having an their home in the subdomain $\omega$.

\medskip 

This motion speed of individuals in a city depends on their home location $y$ of individuals. The distance to individuals' workplaces often relies on their home location in the city. For example, people living in the suburbs travel much longer than people living downtown. Therefore, the traveling speed $\mathbf{C_y}(x)$ at $x \in \R^2$ depends on the home location $y$.  

\section{Model describing the motion of travelers}
\label{Section3}
The convection terms describes the tendency to moving the speed $ \mathbf{C_y}(x)$ at the location $x$ when they started from the home located at the position  $y$. The diffusion describe a random movement around the tendency corresponding to the  convection. In this model the displacement of individuals is described by 
\begin{equation}   \label{3.1}
	\begin{array}{ll}
		\partial_t v_y(t,x)&= \underset{\text{	\begin{tabular}{@{}c@{}}  Random \\motion 	\end{tabular} }}{\underbrace{ \varepsilon^2 \bigtriangleup_x v_y(t,x)}}- \underset{\text{
				\begin{tabular}{@{}c@{}}  
					Deterministic\\movement\\ with speed $\mathbf{C}$
		\end{tabular}}}{\underbrace{ \mathbf{\nabla}_x \cdot  \left(  v_y(t,x) \, \mathbf{C_y}(x) 	\right),}} 
	\end{array}	
\end{equation}
where $\varepsilon^2 \geq 0$ is the diffusion constant (which corresponds to the standard deviation of the law of displacement after one day of the around the original location), and $x \to \mathbf{C_y}(x)=\mathbf{C}(x,y) \in \R^2 $ is a deterministic  speed displacement at location $ x \in \R^2$ for individuals having their home located at the position $y \in \R^2$. 

\medskip 
In this section , we use semigroup theory to define the the solution of \eqref{3.1}. We refer to \cite{Amann, Engel-Nagel, S3, Haase, S4, S5, Magal-Ruan, S7, S8, S9,S10} for more result about semigroups generated by diffusive systems. The book of Lunardy provides a	 very detailed presentation for the case $L^p\left(\R^2\right)$ (with $1<p< \infty$). Here, we consider the case $p=1$. 
\subsection{Purely diffusive model} 
In this section, we consider the equation \eqref{3.1} the special case $\mathbf{C_y}(x) \equiv 0$. That is, 
\begin{equation}   \label{3.2}
	\left\{ 
	\begin{array}{rl}
		\partial_t v(t,x)&=  \varepsilon^2 \bigtriangleup_x v(t,x), \vspace{0.3cm} \\
		v(0,x)&=v_{0}(x) \in \Li^1\left(\R^2\right). 
	\end{array}	
	\right.
\end{equation}
We consider the family of bounded linear operator $\left\{	T_{ \varepsilon^2 \bigtriangleup_x  }(t) \right\}_{t \geq 0} \subset \L\left(\Li^1 \left(\R^2\right)\right)$ defined by 
\begin{equation*}  
	T_{ \varepsilon^2 \bigtriangleup_x  }(t)\left( v(.) \right) (x)=
	\left\{
	\begin{array}{ll} 
		\displaystyle	\int_{\mathbb R^2} K(t,x-z)  v(z)d z, &\text{ for } t>0, \vspace{0.3cm}\\
		v(x),  &\text{ for } t=0,
	\end{array}  
	\right.
\end{equation*}
with
\begin{equation*}   
	K(t,x)=  \dfrac{1}{4\pi\varepsilon^2 t} e^{-\frac{|x|^2}{4\varepsilon^2t}}.
\end{equation*}	
The family of bounded linear operator  $\left\{	T_{ \varepsilon^2 \bigtriangleup_x  }(t) \right\}_{t \geq 0} \subset \L\left(\Li^1 \left(\R^2\right)\right)$  is a strongly continuous semigroup on $\Li^1\left(\R^2\right)$. That is, 
\begin{itemize}
	\item[(i)] $T_{ \varepsilon^2 \bigtriangleup_x  }(0)=I;$
	\item[(ii)] $T_{ \varepsilon^2 \bigtriangleup_x  }(t)T_{ \varepsilon^2 \bigtriangleup_x  }(s)=T_{ \varepsilon^2 \bigtriangleup_x  }(t+s), \forall t,s \geq 0;$
	\item[(iii)] $t \mapsto T_{ \varepsilon^2 \bigtriangleup_x  }(t)u$ is continuous from $[0, +\infty)$ to $\Li^1\left(\R^2\right)$. 
\end{itemize}
Furthermore, $\left\{	T_{ \varepsilon^2 \bigtriangleup_x  }(t) \right\}_{t \geq 0} \subset \L\left(\Li^1 \left(\R^2\right)\right)$  is a semigroup of contraction 
$$
\Vert T_{ \varepsilon^2 \bigtriangleup_x  }(t)  \left(\phi \right) \Vert_{ \Li^1\left(\R^2\right) }\leq \Vert \phi \Vert_{ \Li^1\left(\R^2\right) }, \forall t \geq 0, \forall  \phi \in \Li^1\left(\R^2\right), 
$$ 
and $\left\{	T_{ \varepsilon^2 \bigtriangleup_x  }(t) \right\}_{t \geq 0} \subset \L\left(\Li^1 \left(\R^2\right)\right)$   is a positive semigroup, that is  
\begin{equation} \label{3.3}
	T_{ \varepsilon^2 \bigtriangleup_x  }(t) \bigg(\Li^1_+\left(\R^2\right)  \bigg)   \subset \Li^1_+\left(\R^2\right)  , \forall t \geq 0,  
\end{equation}
and the total of mass of individuals in preserved 
\begin{equation} \label{3.4}
	\int_{\R^2}  T_{ \varepsilon^2 \bigtriangleup_x  }(t)  \left(\phi \right) (x) \dx= \int_{\R^2}  \phi(x) \d x, \forall t \geq 0, \forall  \phi \in \Li^1_+\left(\R^2\right).	
\end{equation}
By using the semigroup property of $\left\{	T_{ \varepsilon^2 \bigtriangleup_x  }(t) \right\}_{t \geq 0} \subset \L\left(\Li^1 \left(\R^2\right)\right)$, we deduces that the family of linear operator 
$$
R_\lambda = \int_0^\infty e^{-\lambda t } 	T_{ \varepsilon^2 \bigtriangleup_x  }(t)\dt, \forall \lambda \in \C, \text{ with } \Re \lambda>0, 
$$
is a pseudo resolvent.   That is 
$$
R_\lambda -R_\mu= \left( \mu-\lambda \right)R_\lambda R_\mu, \forall \lambda, \mu \in \C, \text{ with } \Re \lambda>0. 
$$
From  Lemma 2.2.13. in \cite{Magal-Ruan}, we know that the null space $\Nu (R_\lambda )$ and the range $\Ra(R_\lambda)$ are independent of $\lambda \in \C$ with $\Re \lambda >0$, and the null space  $\Nu (R_\lambda )$ is closed in $\Li^1(\R^2)$. Moreover, by using the strong continuity of the semigroup $\left\{	T_{ \varepsilon^2 \bigtriangleup_x  }(t) \right\}_{t \geq 0} \subset \L\left(\Li^1 \left(\R^2\right)\right)$, one can prove that 
$$
\lambda \, R_\lambda u\to u, \text{ as } \lambda \to + \infty,  
$$
hence 
$$
\Nu (R_\lambda )=\left\{0_{\Li^1}\right\}, \forall \lambda \in \C, \text{ with } \Re \lambda>0. 
$$
Consequenlty,  it follows from  \cite[Proposition 2.2.14]{Magal-Ruan} that there exists a linear closed operator $A:D(A) \subset \Li^1 \left(\R^2\right) \to \Li^1 \left(\R^2\right)$, such that 
$$
R_\lambda =\left(\lambda I -A\right)^{-1}, \forall  \lambda \in \C, \text{ with } \Re \lambda>0. 
$$
Moreover $A$ is the infinitesimal generator of $\left\{	T_{ \varepsilon^2 \bigtriangleup_x  }(t) \right\}_{t \geq 0} \subset \L\left(\Li^1 \left(\R^2\right)\right)$. That is 
$$
D(A)= \left\{ u \in  \Li^1 \left(\R^2\right) : \lim_{t \searrow 0} \dfrac{T_{ \varepsilon^2 \bigtriangleup_x  }(t)u-u}{t} \text{ exists in }  \Li^1 \left(\R^2\right) \right\},
$$
and 
$$
Au= \lim_{t \searrow 0} \dfrac{T_{ \varepsilon^2 \bigtriangleup_x  }(t)u-u}{t}, \forall u \in D(A).
$$
To connect $A$ and $\varepsilon^2  \bigtriangleup_x$, one can prove that 
$$
\lim_{t \searrow 0} \dfrac{T_{ \varepsilon^2 \bigtriangleup_x  }(t)u-u}{t}=\varepsilon^2  \bigtriangleup_x  u, \forall u \in C^2_c \left(\R^2\right).
$$
where $C^2_c \left(\R^2\right)$ is the space of $C^2$ functions with compact support. 

It follows that 
$$ 
C^2_c \left(\R^2\right) \subset D(A),
$$ 
and 
$$
Au=\varepsilon^2  \bigtriangleup_x  u, \forall u \in C^2_c \left(\R^2\right).
$$
Since  $ C^2_c \left(\R^2\right)$ is dense in $ \Li^1 \left(\R^2\right)$,  it follows that the graph of $A$ is the closure of the graph of $\varepsilon^2  \bigtriangleup_x$ considered a linear operator from $ C^2_c \left(\R^2\right)$ into $ \Li^1 \left(\R^2\right)$. 

\begin{remark}  \label{RE3.1}
	In the above problem, the difficulty is to define the domain $D(A)$ of $A$ properly. This domain is not explicit in dimension $2$, and the goal is to guarantee the invertibility of  $\lambda I -A$ from $D(A)$ to $  \Li^1 \left(\R^2\right) $. The Proposition 8.1.3 p. 223 in the book Haase \cite{Haase} gives 
	$$
	W^{1,1}\left(\R^2\right) \subset D(A) \subset W^{2,1}\left(\R^2\right). 
	$$
\end{remark}


\begin{lemma}  \label{LE3.2} The semigroup $\left\{	T_{ \varepsilon^2 \bigtriangleup_x  }(t) \right\}_{t \geq 0} \subset \L\left(\Li^1 \left(\R^2\right)\right)$   is irreducible. That is,    for each $u \in L^1_+\left(\R^2\right)$ with $u \neq 0$, and each $\phi \in L^\infty_+\left(\R^2\right)$ with $\phi \neq 0$, 
	$$
	\int_{\R^2} \phi(x) T_{ \varepsilon^2 \bigtriangleup_x  }(t)(u)(x) dx >0, \forall t >0. 
	$$ 
\end{lemma} 

\begin{proof} Let $u \in L^1_+\left(\R^2\right)$ with $u \neq 0$,  $ \phi \in L^\infty_+\left(\R^2\right)$ with $ \phi \neq 0$, and $t>0$.  By using Fubini theorem,  we have 
	$$
	\begin{array}{ll}
		\displaystyle	\int_{\R^2}  \phi(x)  T_{ \varepsilon^2 \bigtriangleup_x  }(t)(u)(x)dx &  = 	\displaystyle		\int_{\R^2}  \phi(x) 	\int_{\mathbb R^2} K(t,x-z)  u(z)d z  dx	
	\end{array}	
	$$
	and since 
	\begin{equation*}   
		K(t,x)=  \dfrac{1}{4\pi\varepsilon^2 t} e^{-\frac{|x|^2}{4\varepsilon^2t}}.
	\end{equation*}	
	it follows that $x \to \int_{\mathbb R^2}    K(t,x-x) u(x) dx $ is continuous and strictly positive for each $x \in \R^2$.  The result follows. 
\end{proof}
\subsection{Purely convective model} 
In this section, we consider the equation \eqref{3.1} the special case $\varepsilon=0$. That is, 
\begin{equation}  \label{3.5}
	\left\{ 
	\begin{array}{rl}
		\partial_t v(t,x)&=- \mathbf{\nabla}_x \cdot  \left(  v(t,x) \, \mathbf{C_y}(x) 	\right), \vspace{0.3cm} \\
		v(0,x)&=v_{0}(x) \in \Li^1\left(\R^2\right). 
	\end{array}	
	\right.
\end{equation}
To define the solutions integrated along the characteristics we make the following assumptions. 
\begin{assumption} \label{ASS3.3}
	Let $\mathbf{C}:\R^2 \to \R^2$ be a maps. We assume that 
	\begin{itemize}
		\item[{\rm (i)}] The map $x \in \R^2 \mapsto  \mathbf{C}(x)\in \R^2 $ is uniformly continuous bounded; 
		\item[{\rm (ii)}] The map $x \in \R^2 \mapsto  \mathbf{C}(x) \in \R^2 $ is supposed to be a $C^1$ function;
		\item[{\rm (iii)}] 	For $i=1,2$, the map 	$x \mapsto \partial_{x_i} \mathbf{C}(x)$ is bounded and uniformly continuous. 
	\end{itemize}
\end{assumption}
Assume that $\mathbf{C_y}$ satisfies the above assumption. Then the map  $x \in \R^2 \mapsto  \mathbf{C_y}(x)\in \R^2 $ is Lipschitz continuous, and the flow on $\R^2$ generated by 
\begin{equation} \label{3.6}
	\left\{ 
	\begin{array}{l}
		\partial_t \Pi_y(t)z=\mathbf{C_y}\left(\Pi_y(t)z\right), \forall t \in \R,  \vspace{0.2cm}\\
		\Pi_y(0)z=z \in \R^2,
	\end{array}
	\right.
\end{equation}
is well defined. Moreover we have the following property. 
\begin{lemma} \label{LE3.4} 	Assume that $\mathbf{C_y}$ satisfies Assumption \ref{ASS3.3}.We have 
	\begin{equation} \label{3.7}
		\det  \partial_z \Pi_y(t)z= \exp\left( \int_{0}^{t} \mathbf{\nabla}_x \cdot  \mathbf{C_y}\left(\Pi_y(\sigma )z\right)  \d \sigma  \right), \forall t \geq 0, 
	\end{equation}
	and 
	\begin{equation} \label{3.8}
		\det  \partial_z \Pi_y(-t)z= \exp\left(- \int_{0}^{t} \mathbf{\nabla}_x \cdot  \mathbf{C_y}\left(\Pi_y(-\sigma )z\right)  \d \sigma  \right), \forall t \geq 0.
	\end{equation}
\end{lemma}
\begin{proof}
	Define $U(t):=\partial_z \Pi_y(t)z \in M_n \left(\R\right)$. We know that 
	$$
	\dfrac{dU(t)}{dt} = \nabla_x \mathbf{C_y}\left(\Pi_y(t)z\right) U(t), \text{ and } U(0)=I. 
	$$
	For any matrix-valued $C^1$  function $A : t  \mapsto A(t)$  the Jacobi's formula reads
	$$
	\dfrac{d}{dt}  \det A(t) = \det A(t) \, {\rm tr} (A^{-1}(t) \dfrac{dA(t)}{dt}) 
	$$
	and by using the property of the trace $ {\rm tr} \left(A B\right)= {\rm tr} \left(BA\right) $, we deduce that 
	$$
	\dfrac{d}{dt}  \det U(t) = \det  U(t) \, {\rm tr} (	\dfrac{dU(t)}{dt}  \, U(t)^{-1} ) =\det  U(t) \, {\rm tr} (	 \nabla_x \mathbf{C_y}\left(\Pi_y(t)z\right)  ) 
	$$
	and the result follows from the fact that 
	$$
	{\rm tr} (	 \nabla_x \mathbf{C_y}\left(\Pi_y(t)z\right)  ) = \mathbf{\nabla}_x \cdot  \mathbf{C_y}\left(\Pi_y(t)z\right)  . 
	$$
	Consider now $\widehat{\Pi}_y(t)z= \Pi_y(-t)z$. Then 
	\begin{equation} \label{3.9}
		\left\{ 
		\begin{array}{l}
			\partial_t \widehat{\Pi}_y(t)z=-\mathbf{C_y}\left(\widehat{\Pi}_y(t)z\right), \forall t \in \R,  \vspace{0.2cm}\\
			\widehat{\Pi}_y(0)z=z \in \R^2.
		\end{array}
		\right.
	\end{equation}
	Therefore \eqref{3.8} follows from \eqref{3.7}. 
\end{proof}

\medskip 
Assume first that the solution of \eqref{3.3} is $C^1$. That is 
$$
v \in C^1\left( \R \times \R^2, \R \right). 
$$
Then the right hand side of \eqref{3.3} can be expended, and \eqref{3.3} reads as 
\begin{equation*} 
	\partial_t v(t,x)=- \mathbf{C_y}(x) \cdot  \mathbf{\nabla}_x  v(t,x) -v(t,x) \,  \mathbf{\nabla}_x \cdot   \mathbf{C_y}(x), 	
\end{equation*}
where $ \mathbf{\nabla}_x  v(x) $ is the gradient of $x \mapsto v(x)$ which is defined by 
$$
\mathbf{\nabla}_x \,  v(x)=\left( \begin{array}{c}
	\partial_{x_1} v(x)\\
	\partial_{x_2} v(x)
\end{array}\right). 
$$
Moreover, we have 
\begin{equation*}
	\begin{array}{ll}
		\dfrac{d}{dt}v(t,\Pi_y(t)z)&=\partial_t v(t,\Pi_y(t)z)+ \nabla_x  v(t,\Pi_y(t)z) \cdot \partial_t \Pi_y(t)z \vspace{0.2cm}\\
		&= -\mathbf{C_y}(\Pi_y(t)z) \cdot  \mathbf{\nabla}_x  v(t,\Pi_y(t)z) -v(t,\Pi_y(t)z) \,  \mathbf{\nabla}_x \cdot   \mathbf{C_y}(\Pi_y(t)z) \vspace{0.2cm}\\
		& \quad + \nabla_x  v(t,\Pi_y(t)z) \cdot  \mathbf{C_y}\left(\Pi_y(t)z\right),
	\end{array}
\end{equation*}
and we obtain 
\begin{equation*}
	\dfrac{d}{dt}v(t,\Pi_y(t)z)= -v(t,\Pi_y(t)z) \,  \mathbf{\nabla}_x \cdot   \mathbf{C_y}(\Pi_y(t)z) \vspace{0.2cm}\\
\end{equation*}
Therefore 
\begin{equation*}
	v(t,\Pi_y(t)z)=\exp \left(- \int_{0}^{t}  \mathbf{\nabla}_x \cdot   \mathbf{C_y}(\Pi_y(\sigma )z) \, \d \sigma \right) v(0,z)
\end{equation*}
by choosing $z=\Pi_y(-t)x$ we obtain the following explicit formula for the solutions 
\begin{equation*}
	v(t,x)=\exp \left( -\int_{0}^{t}  \mathbf{\nabla}_x \cdot   \mathbf{C_y}(\Pi_y(\sigma-t )x) d \sigma \right) v_{0} \left(\Pi_y(-t)x \right),
\end{equation*}
or equivalently 
\begin{equation*}
	v(t,x)=\exp \left(- \int_{0}^{t}  \mathbf{\nabla}_x \cdot   \mathbf{C_y}(\Pi_y(-\sigma )x) d \sigma \right) v_{0} \left(\Pi_y(-t)x \right). 
\end{equation*}
We consider the family of bounded linear operator $\left\{		T_{ B_y }(t)  \right\}_{t \geq 0} \subset \L\left(\Li^1 \left(\R^2\right)\right)$ defined by 
\begin{equation}  \label{3.10}
	T_{ B_y }(t)\left( v_{0} \right) (x)=\exp \left(- \int_{0}^{t}  \mathbf{\nabla}_x \cdot   \mathbf{C_y}(\Pi_y(-\sigma )x) d \sigma  \right) v_{0} \left(\Pi_y(-t)x \right). 
\end{equation}
Similarly to the diffusion we also have the following result. 
\begin{lemma} \label{LE3.5}
	Assume that $\mathbf{C_y}$ satisfies Assumption \ref{ASS3.3}.	There exists a closed linear operator $B_y:D(B_y) \in \Li^1\left(\R^2\right) \to \Li^1\left(\R^2\right)$ the infinitesimal generator of a strongly continuous semigroup  $\left\{	T_{B_y  }(t) \right\}_{t \geq 0} \subset \L\left(\Li^1 \left(\R^2\right)\right)$ of positive bounded linear operator on $\Li^1 \left(\R^2\right)$ defined by \eqref{3.10}. 
\end{lemma}

We observe that we have the following conservation of the number of individuals is preserved. That is, for each Borelian set $\Omega \subset \R^2$, 
$$
\int_{\Omega}	T_{ B_y }(t)\left( v_{0} \right) (x)=\int_{\Omega}	\exp \left(- \int_{0}^{t}  \mathbf{\nabla}_x \cdot   \mathbf{C_y}(\Pi_y(-\sigma )x) d \sigma  \right) v_{0} \left(\Pi_y(-t)x \right)dx 
$$
and by using \eqref{3.8}, we obtain 
$$
\int_{\Omega}	T_{ B_y }(t)\left( v_{0} \right) (x)=\int_{\Omega}	 v_{0} \left(\Pi_y(-t)x \right) \det  \partial_z \Pi_y(-t)x dx,
$$ 
therefore by making a change of variable $z=\Pi_y(-t)x $, we obtain 
\begin{equation*}
	\int_{\Omega}	T_{ B_y }(t)\left( v_{0} \right) (x)=\int_{\Pi_y(-t) \Omega}	 v_{0} \left( z \right) dz, \forall t \geq 0. 
\end{equation*}
When $\Omega= \R^2$, we deduce that the total mass of individuals is preserved. That is, 
\begin{equation} \label{3.11}
	\int_{\R^2}	T_{ B_y }(t)\left( v_{0} \right) (x)dx=\int_{\R^2}		 v_{0} \left( x \right) dx, \forall t \geq 0. 
\end{equation}
By using the semi-explicitly formula \eqref{3.10} that $\left\{	T_{B_y  }(t) \right\}_{t \geq 0} \subset \L\left(\Li^1 \left(\R^2\right)\right)$  is a strongly continuous semigroup on $\Li^1\left(\R^2\right)$, and 
\begin{equation}  \label{3.12}
	\Vert T_{B_y  }(t) v_{0} \Vert_{ \Li^1\left(\R^2\right) }=	\Vert v_{0} \Vert_{ \Li^1\left(\R^2\right) }, \forall t \geq 0. 
\end{equation}
Moreover, one has
$$
\lim_{t \searrow 0} \dfrac{T_{B_y   }(t)v_{0}-v_{0}}{t}=- \mathbf{\nabla}_x \cdot  \left(  v_{0}(x) \, \mathbf{C_y}(x) 	\right), \forall v_{0} \in C^1 \left(\R^2\right) \cap W^{1,1}\left(\R^2\right), 
$$
where 
$$
C^1 \left(\R^2\right) \cap W^{1,1}\left(\R^2\right)= \left\{ v \in  C^1 \left(\R^2\right) \cap \Li^{1}\left(\R^2\right): x \mapsto \partial_{x_i} v(x) \in  \Li^{1}\left(\R^2\right), \forall i=1,2  \right\}.
$$
It follows that,
$$ 
C^1_c \left(\R^2\right) \subset  C^1 \left(\R^2\right) \cap W^{1,1}\left(\R^2\right) \subset D(B_y),
$$ 
where $C^1_c \left(\R^2\right)$ is the space of $C^1$ with compact support. 

\medskip 
Moreover, 
$$
B_y v=- \mathbf{\nabla}_x \cdot  \left(  v(x) \, \mathbf{C_y}(x) 	\right), \forall v \in C^1 \left(\R^2\right) \cap W^{1,1}\left(\R^2\right) ,
$$
and since $ C^1_c \left(\R^2\right)$ is dense in $ \Li^1 \left(\R^2\right)$,  it follows that the graph of $B_y$ is the closure of the graph of $v \mapsto - \mathbf{\nabla}_x \cdot  \left(  v(x) \, \mathbf{C_y}(x) 	\right)$ considered a linear operator from $ C^1_c \left(\R^2\right)$ into $ \Li^1 \left(\R^2\right)$. 
\subsection{Existence of mild solutions for the full problem with both diffusion and convection}
In this section, we consider the full equation \eqref{3.1}
\begin{equation}  \label{3.13}
	\left\{ 
	\begin{array}{rl}
		\partial_t v(t,x)&= \varepsilon^2 \bigtriangleup_x v(t,x) - \mathbf{\nabla}_x \cdot  \left(  v(t,x) \, \mathbf{C_y}(x) 	\right), \vspace{0.2cm} \\
		v(0,x)&=v_{0}(x) \in \Li^1\left(\R^2\right). 
	\end{array}	
	\right.
\end{equation}
By using the notations introduced in the previous sections, this problem rewrites as the following abstract Cauchy problem 
\begin{equation} \label{3.14}
	\left\{ 
	\begin{array}{l}
		v'(t)=(A+B_y)v(t), \text{ for } t \geq 0,  \vspace{0.2cm} \\
		v(0)=v_0 \in \Li^1\left(\R^2\right). 	
	\end{array}
	\right.
\end{equation}
In order to define the mild solutions of \eqref{3.13} as a continuous function $t \in \left[0, \infty\right) \mapsto v(t) \in \Li^1\left(\R^2\right),$ a mild solution 
\begin{equation} 
	\label{3.15}
	v(t)=T_{ \varepsilon^2 \bigtriangleup_x  }(t)v_0+ \int_0^t T_{ \varepsilon^2 \bigtriangleup_x  }(t-\sigma)B_y \,v(\sigma ) \d\sigma.
\end{equation}
The existence of the solutions follows by considering the following system 
\begin{equation}  \label{3.16}
	\left\{ 
	\begin{array}{l}
		v(t)=T_{ \varepsilon^2 \bigtriangleup_x  }(t)v_0+ \int_0^t T_{ \varepsilon^2 \bigtriangleup_x  }(t-\sigma)w(\sigma ) \d\sigma, \vspace{0.2cm} \\
		w(t)=B_yT_{ \varepsilon^2 \bigtriangleup_x  }(t)v_0+ \int_0^t B_yT_{ \varepsilon^2 \bigtriangleup_x  }(t-\sigma)w(\sigma ) \d\sigma.
	\end{array}
	\right.
\end{equation}
We observe that 
\begin{equation} \label{3.17}
	\mathbf{\nabla}_x \cdot  \left( w(t,x) \, \mathbf{\widehat{C}_y}(x) \right) = \mathbf{C}_y(x) \cdot  \mathbf{\nabla}_x w(t,x) +w(t,x) \,  \mathbf{\nabla}_x \cdot   \mathbf{C_y}(x), 	
\end{equation}
where $ \mathbf{\nabla}_x  w(t,x) $ is the gradient of $x \mapsto w(t,x)$ which is defined by 
$$
\mathbf{\nabla}_x \,  w(t,x)=\left( \begin{array}{c}
	\partial_{z_1} w(t,x)\\
	\partial_{z_2} w(t,x)
\end{array}\right). 
$$

\begin{lemma} \label{LE3.6} 	Assume that $\mathbf{C_y}$ satisfies Assumption \ref{ASS3.3}. Let   $ y \in \R^2$. There exists a constant $\kappa>0$ such that   for each $ u \in \Li^1 \left(\R^2\right)$, 
	\begin{equation}  \label{3.18}
		T_{ \varepsilon^2 \bigtriangleup_z  }(t) u \subset  C^1 \left(\R^2\right) \cap W^{1,1}\left(\R^2\right) \subset D \left( B_y \right), \forall t>0,  
	\end{equation}
	and 
	\begin{equation} \label{3.19}
		\Vert B_{y}  	T_{ \varepsilon^2 \bigtriangleup_z  }(t) u \Vert_{\Li^1 \left(\R^2 \right)}  \leq  \kappa \left( \dfrac{1}{\sqrt{\varepsilon^2t}} +1  \right)  \Vert u \Vert_{\Li^1 \left(\R^2 \right) }, \forall t>0 . 
	\end{equation}
\end{lemma}
\begin{proof} We observe that 
	\begin{equation*}   
		K(t,x)=  \dfrac{1}{4\pi\varepsilon^2 t} e^{-\frac{x_1^2+x_2^2}{4\varepsilon^2t}}=	K_1(t,x_1)	K_1(t,x_2).  
	\end{equation*}
	where 
	\begin{equation*}   
		K_1(t,x)=  \dfrac{1}{\sqrt{4\pi \varepsilon^2t}} e^{-\frac{x^2}{4\varepsilon^2t}} \,.
	\end{equation*} 
	Moreover 
	$$
	\int_{\R} \vert \partial_{x} 	K_1(t,x) \vert dx = \dfrac{2}{\sqrt{4\varepsilon^2t}}  \int_0^\infty \frac{2x}{4\varepsilon^2t} e^{-\frac{x^2}{4\varepsilon^2t}}dx=\dfrac{1}{\sqrt{\varepsilon^2t}} . 
	$$
	We observe that 
	$$
	\begin{array}{ll}
		\partial_{x_1}  	T_{ \varepsilon^2 \bigtriangleup_x  }(t) \left( u  \right)	\left(x_1,x_2 \right)  = \int_{\mathbb R}   \partial_{x} K_1(t,x_1-\sigma_1)  \int_{\mathbb R}   K_1(t,x_2-\sigma_2) u (\sigma_1,\sigma_2) d\sigma_2 d \sigma_1,
	\end{array}
	$$
	and it follows 
	$$
	\begin{array}{ll}
		\Vert	\partial_{x_1}  	T_{ \varepsilon^2 \bigtriangleup_x  }(t) \left( u  \right)	 \Vert_{ \Li^1\left(\R^2\right) } & \leq \displaystyle \int_{\R} \vert \partial_{x} 	K_1(t,x) \vert dx \vspace{0.2cm} \\
		& \quad  \displaystyle \times \int_{\mathbb R}   \int_{\mathbb R}   K_1(t,x_2-\sigma_2) u (\sigma_1,\sigma_2) d\sigma_2 d \sigma_1 , 
	\end{array}
	$$  
	and the proof is completed. 
\end{proof}
By using the previous we deduce the following result.  
\begin{proposition}  \label{PROP3.7}	Assume that $\mathbf{C_y}$ satisfies Assumption \ref{ASS3.3}. We have 
	\begin{equation} \label{3.20}
		D\left(B_y\right) \subset  D(A) , 	
	\end{equation}
	and 
	\begin{equation} \label{3.21}
		\Vert B_y \left(\lambda I -A\right)^{-1} \Vert_{\L\left( \Li^1\left(\R^2\right) \right) } \leq \int_{0}^{\infty}  e^{-\lambda t }  \kappa \left( \dfrac{1}{\sqrt{\varepsilon^2t}} +1  \right)   \d t , \forall \lambda >0.	
	\end{equation}
\end{proposition}
\begin{proof}
	Let $\lambda >0$.  We have 
	$$
	D(A)=	\left(\lambda I- A \right)^{-1}	\Li^1\left( \R^2 \right),
	$$
	and
	$$
	\left(\lambda I- A \right)^{-1}u=\int_{0}^{\infty}  e^{-\lambda t }	T_{ \varepsilon^2 \bigtriangleup_z  }(t) u  \d t, \forall u \in \Li^1\left( \R^2 \right). 
	$$
	Since $B_y$ is a closed linear operator, we have  
	$$
	B_y 	\left(\lambda I- A \right)^{-1}u =\int_{0}^{\infty}  e^{-\lambda t } B_y	T_{ \varepsilon^2 \bigtriangleup_z  }(t) u   \d t , \forall u \in \Li^1\left( \R^2 \right),
	$$
	and by \eqref{3.19} we deduce that the right hand-side of the above equality is integrable, and \eqref{3.20} follows, and  
	$$
	\Vert  	B_y 	\left(\lambda I- A \right)^{-1}u \Vert_{\Li^1 \left(\R^2 \right) } \leq \int_{0}^{\infty}  e^{-\lambda t }  \kappa \left( \dfrac{1}{\sqrt{\varepsilon^2t}} +1  \right)   \d t \Vert u \Vert_{\Li^1 \left(\R^2 \right) } . 
	$$ 
\end{proof}

\medskip 
Since $D(B_y) \subset D(A)$, we can  $\left(A+B_y\right) : D(A) \subset  \Li^1\left(\R^2\right) \to  \Li^1\left(\R^2\right)$ is well defined by 
$$
\left(A+B_y\right) u=Au+B_y u, \forall u \in D(A). 
$$
\begin{definition}  \label{DE3.8}
	We will say that a continuous map $u \in C \left( [0,\infty),  \Li^1\left(\R^2\right)\right) $ is a\textbf{ mild solution} of \eqref{3.14} if and only if 
	$$
	\int_{0}^{t} v(s)ds  \in D(A), \forall t \geq 0,
	$$
	and
	$$
	v(t)=v_0 + \left(A+B_y\right) \int_{0}^{t} v(s)ds.
	$$
\end{definition}

We observe that 
$$
K\left(\alpha \right) = \kappa \int_{0}^{\infty} e^{-\alpha \sigma} \left(  \dfrac{1}{\sqrt{\varepsilon^2t}} +1  \right)  \d \sigma  <\infty, \forall \alpha >0, 
$$
and 
$$
\lim_{ \alpha \to+\infty} K\left(\alpha \right) =0. 
$$
We consider the weighted space of integrable function $ L^1_\alpha  \left( \left( 0,\infty \right); \Li^1\left(\R^2\right) \right)$ which is the space of Bochner measurable function $t \mapsto f(t)$ from $(0, \infty)$ to $ \Li^1\left(\R^2\right) $ satisfying 
$$
\int_{0}^{\infty} e^{-\alpha t} \Vert f (t)\Vert_{ \Li^1\left(\R^2\right) } \dt <+\infty.
$$
Then $ L^1_\alpha  \left( \left( 0,\infty \right); \Li^1\left(\R^2\right) \right)$ is a Banach space endowed with the norm 
$$
\Vert f \Vert_{ L^1_\alpha } = \int_{0}^{\infty} e^{-\alpha t} \Vert f (t)\Vert_{ \Li^1\left(\R^2\right) } \dt . 
$$
Let $\alpha_0>0$ such that $K\left(\alpha_0 \right) <1$. Then for each $v_0 \in \Li^1\left(\R^2\right)$, by applying the Banach fixed theorem, we deduce that there exists a unique solution $ w \in L^1_{\alpha_0}  \left( \left( 0,\infty \right); \Li^1\left(\R^2\right) \right) $ satisfying the fixed point problem 
\begin{equation} \label{3.22}
	w(t)=B_yT_{ \varepsilon^2 \bigtriangleup_x  }(t)v_0+ \int_0^t B_yT_{ \varepsilon^2 \bigtriangleup_x  }(t-\sigma)w(\sigma ) \d\sigma.	 
\end{equation}
By using the same arguments as in   Ducrot, Magal, and Prevost \cite[Theorem 4.8]{DPM}, we obtain the following result. 
\begin{theorem} \label{TH3.9}
	Assume that $\mathbf{C_y}$ satisfies Assumption \ref{ASS3.3}.	 Let $\alpha_0>0$ such that $K\left(\alpha_0 \right) <1$. The linear operator $\left(A+B_y\right): D(A) \subset  \Li^1\left(\R^2\right) \to  \Li^1\left(\R^2\right)$ is the infinitesimal generator of an analytic semigroup. Moreover, for each $v_0 \in \Li^1\left(\R^2\right)$, the Cauchy problem  \eqref{3.14} admits a unique mild solution $t \to T_{A+B_y}(t)v_0$. Furthermore, the map $t \to v(t)= T_{A+B_y}(t)v_0$ satisfies 
	$$
	v(t)=T_{ \varepsilon^2 \bigtriangleup_x  }(t)v_0+ \int_0^t T_{ \varepsilon^2 \bigtriangleup_x  }(t-\sigma)w(\sigma ) \d\sigma, \forall t \geq 0,	 
	$$
	where $ w \in L^1_{\alpha_0}  \left( \left( 0,\infty \right); \Li^1\left(\R^2\right) \right) $  is the unique solution the fixed point problem \eqref{3.22}.
\end{theorem}

\medskip 
Let $\lambda \geq \alpha_0$. Since $B_y$ is a closed linear operator, we have 
$$
B_y  \left( \lambda I - A \right)^{-1}=\int_{0}^{\infty}e^{-\lambda t} B_y  T_{ \varepsilon^2 \bigtriangleup_x  }(t)dt, 
$$
and 
$$
\Vert B_y  \left( \lambda I - A \right)^{-1} \Vert_{\L \left(L^1 \left(\R^2\right) \right)} \leq  \kappa \int_{0}^{\infty} e^{-\lambda \sigma} \left(  \dfrac{1}{\sqrt{\varepsilon^2t}} +1  \right)  \d \sigma= K\left(\lambda \right) \leq K\left(\alpha_0\right)  <1. 
$$
Let $u \in D(A)$ and $v \in  \Li^1\left(\R^2\right)$. We have 
$$
\begin{array}{rl}
	\left( \lambda I -A-B_y \right)u=v  \Leftrightarrow& \left[ I -B_y  \left( \lambda I - A \right)^{-1}  \right]   \left( \lambda I - A \right)u=v \\
	\Leftrightarrow&u=  \left( \lambda I - A \right)^{-1} \left[ I -B_y  \left( \lambda I - A \right)^{-1}  \right]^{-1}v 
\end{array}
$$ 
We obtain the following lemma. 
\begin{lemma} \label{LE3.10}	Let Assumption \ref{ASS3.3} be satisfied. We have   
	$$
	(\alpha_0 ,+\infty ) \subset  \rho (A+B_y),
	$$ 
	the resolvent set of $A+B_y$, and for each $\lambda > \alpha_0$, 
	$$
	\left( \lambda I -A-B_y \right)u=v  \Leftrightarrow	u=  \left( \lambda I - A \right)^{-1} \sum_{k \geq 0} \left[ B_y  \left( \lambda I - A \right)^{-1}  \right]^{k}v . 
	$$
\end{lemma}

\subsection{Positivity of the solutions  for the full problem  with both diffusion and convection}
In this section, we reconsider the positivity of the solutions by using only abstract argument.  Such a problem was study by Protter and Weinberger \cite{Protter-Weinberger} by using maximum principle. Here we use the fact that $A$ and $B_y$ are both the infinitesimal generator of positive semi-groups, together with some suitable estimation on $B_y T_{A}(t), \forall t>0$.  

\medskip 
Recall that the Hille-Yosida approximation of $B_y$ is defined by  
\begin{equation} \label{3.23}
	B^\lambda_y=\lambda B_y\left(\lambda I - B_y\right)^{-1} , \forall \lambda >0. 	
\end{equation}
Then we have 
\begin{equation} \label{3.24}
	B^\lambda_y=- \lambda I+ \lambda^2 \left(\lambda I - B_y\right)^{-1} , \forall \lambda >0. 	
\end{equation}
Recall that 
$$
\lim_{\lambda \to + \infty} \lambda \left(\lambda I - B_y\right)^{-1}u =u,\, \forall u \in \Li^1\left(\R^2\right), 
$$
we deduce that   
$$
\lim_{\lambda \to + \infty} B^\lambda_y u = B_y u,\, \forall u \in D\left(B_y\right). 
$$
The idea of this section is to approximate the problem \eqref{3.15}
\begin{equation*} 
	v(t)=T_{ \varepsilon^2 \bigtriangleup_x  }(t)v_0+ \int_0^t T_{ \varepsilon^2 \bigtriangleup_x  }(t-\sigma)B_y v(\sigma ) \d\sigma.
\end{equation*}
by using the Hille-Yosida approximation of $B_y$. That is, 
\begin{equation*} 
	v_\lambda(t)=T_{ \varepsilon^2 \bigtriangleup_x  }(t)v_0+ \int_0^t T_{ \varepsilon^2 \bigtriangleup_x  }(t-\sigma)B^\lambda_y v_\lambda(\sigma ) \d\sigma.
\end{equation*}
\subsubsection{Convergence of the approximation}
Let $v_0 \in D(A)$. Define 
$$
w_\lambda(t)=B^\lambda_y v_\lambda(t),
$$
which satisfies 
\begin{equation*} 
	w_\lambda(t)=B^\lambda_y T_{ \varepsilon^2 \bigtriangleup_x  }(t)v_0+ \int_0^t B^\lambda_y T_{ \varepsilon^2 \bigtriangleup_x  }(t-\sigma) w_\lambda (\sigma ) \d\sigma.
\end{equation*}
By computing the difference between the above equation and \eqref{3.22}, we obtain 
\begin{equation} \label{3.25}
	\begin{array}{ll}
		w_\lambda(t)-	w(t) &=	 \left[  \lambda \left(\lambda I - B_y\right)^{-1} - I\right] B_y T_{ \varepsilon^2 \bigtriangleup_x  }(t)v_0 \vspace{0.2cm}\\
		&+\displaystyle   \int_0^t \left[  \lambda \left(\lambda I - B_y\right)^{-1} - I\right]  B_y	T_{ \varepsilon^2 \bigtriangleup_x  }(t-\sigma)w(\sigma ) \d\sigma \vspace{0.2cm}\\
		& +\displaystyle  \int_0^t   \lambda \left(\lambda I - B_y\right)^{-1}   B_y  T_{ \varepsilon^2 \bigtriangleup_x  }(t-\sigma)\left( w_\lambda(\sigma )- w( \sigma) \right) \d\sigma,\\	
	\end{array}
\end{equation}
Let $\tau>0$. By using the fact that $t \to  B_y T_{ \varepsilon^2 \bigtriangleup_x  }(t)w_0 $ maps bounded interval $\left[0, \tau \right]$ into a compact subset of $ \Li^1\left(\R^2\right)$, and $(t,\sigma ) \to  B_y T_{ \varepsilon^2 \bigtriangleup_x  }(t-\sigma)w(\sigma)$ maps bounded subsets $\left\{ (t, \sigma) \in \left[0, \tau \right] : t \geq \sigma  \right\}$ into a compact subset of $ \Li^1\left(\R^2\right)$, we deduce that 
\begin{equation}
	\lim_{\lambda \to+\infty}\sup_{t \in [0, \tau]} \Vert \left[  \lambda \left(\lambda I - B_y\right)^{-1} - I\right] B_y T_{ \varepsilon^2 \bigtriangleup_x  }(t)w_0  \Vert_{ \Li^1\left(\R^2\right) } =0, 	
\end{equation}
and 
\begin{equation}
	\lim_{\lambda \to+\infty}\sup_{t \in [0, \tau]} \Vert  \int_0^t \left[  \lambda \left(\lambda I - B_y\right)^{-1} - I\right]  B_yT_{ \varepsilon^2 \bigtriangleup_x  }(t-\sigma)w(\sigma ) \d\sigma \Vert_{ \Li^1\left(\R^2\right) }  =0.	
\end{equation}
Moreover we have 
$$
\Vert \lambda \left(\lambda I - B_y\right)^{-1}  \Vert_{\L\left( \Li^1\left(\R^2\right) \right) }\leq 1,  \forall \lambda >0, 
$$
hence 
\begin{equation}
	\begin{array}{r}
		\displaystyle  	\Vert   \int_0^t   \lambda \left(\lambda I - B_y\right)^{-1} B_y  T_{ \varepsilon^2 \bigtriangleup_x  }(t-\sigma)  \left( w_\lambda(\sigma )- w( \sigma)  \right)\d\sigma \Vert_{ \Li^1\left(\R^2\right) } \vspace{0.2cm}\\
		\leq 	\displaystyle  \kappa \int_{0}^{\tau}  \left(  \dfrac{1}{\sqrt{\varepsilon^2 \sigma}} +1  \right)  \d \sigma \displaystyle \sup_{\sigma \in [0, \tau ]}\Vert  w_\lambda(\sigma )- w( \sigma)\Vert_{ \Li^1\left(\R^2\right) } . 
		
	\end{array}
\end{equation}
\begin{lemma} \label{LE3.11}	Assume that $\mathbf{C_y}$ satisfies Assumption \ref{ASS3.3}.Let $\tau>0$ small enough to satisfy 
	$$
	\kappa \int_{0}^{\tau} \left(  \dfrac{1}{\sqrt{\varepsilon^2 \sigma  }} +1  \right)  \d \sigma <1. 
	$$
	Then for each $v_0 \in D(A)$, and each $\tau >0$, we have 
	$$
	\lim_{\lambda \to+\infty}\sup_{t \in [0, \tau]}\Vert v_\lambda(t)-v(t)\Vert_{ \Li^1\left(\R^2\right) }  =0.
	$$
	
\end{lemma}
\subsubsection{Positivity}
By using \eqref{3.24}, we deduce that 
\begin{equation} 
	\label{3.29}
	v_\lambda(t)=T_{ \varepsilon^2 \bigtriangleup_x-\lambda I  }(t)v_0+ \int_0^t  T_{ \varepsilon^2 \bigtriangleup_x -\lambda I }(t-\sigma)\lambda^2 \left(\lambda I - B_y\right)^{-1}   v_\lambda(\sigma ) \d\sigma,
\end{equation}
where 
$$
T_{ \varepsilon^2 \bigtriangleup_x-\lambda I  }(t)=e^{-\lambda t}T_{ \varepsilon^2 \bigtriangleup_x  }(t). 
$$
If $u_0 \in D(A) \cap \Li^1_+\left(\R^2\right)$, since $ \left(\lambda I - B_y\right)^{-1}  $ is a positive bounded linear operator, we deduce that 
\begin{equation}
	v_\lambda(t)\geq 0, \forall t \geq 0. 	
\end{equation}
To obtain the positivity is sufficient to use the fact that $ D(A) \cap \Li^1_+\left(\R^2\right)$ is dense in $   \Li^1_+\left(\R^2\right)$, which follows from the following observation 
$$
\lambda \left(\lambda I - A\right)^{-1}  v_0\in D(A)\cap  \Li^1_+\left(\R^2\right), \forall v_0 \in  \Li^1_+\left(\R^2\right), \forall \lambda >0, 
$$
and 
$$
\lim_{\lambda \to \infty} \lambda \left(\lambda I - A\right)^{-1}  v_0=v_0, \forall v_0 \in  \Li^1\left(\R^2\right).
$$
By using Lemma \ref{LE3.11}, we obtain the following theorem. 
\begin{theorem}[Positivity] \label{TH3.12} Assume that $\mathbf{C_y}$ satisfies Assumption \ref{ASS3.3}.
	For each $v_0 \in \Li^1_+\left(\R^2\right),$ the solution of Cauchy problem  \eqref{3.14} non-negative. That is 
	\begin{equation}
		T_{A+B_y}(t)v_0 \geq 0, \forall t \geq 0.		
	\end{equation}
\end{theorem}

As a consequence of Theorem \ref{TH3.12}, we obtain an  abstract proof of the result of Protter-Weinberger \cite{Protter-Weinberger}. 
\begin{corollary} Assume that $\mathbf{C}:\R^2 \to \R^2$ satisfies Assumption \ref{ASS3.3}. Let $\chi: \R^2 \to \R $   be bounded and uniformly continuous map. Consider the system  
	\begin{equation}  \label{3.32}
		\left\{ 
		\begin{array}{rl}
			\partial_t v(t,x)&= \varepsilon^2 \bigtriangleup_x v(t,x)  \vspace{0.2cm} \\
			& \quad  - \mathbf{C}(x) _1 \, \partial_{x_1} v(t,x) - \mathbf{C}(x)_2 \, \partial_{x_2} v(t,x)  \vspace{0.2cm} \\
			& \quad + \chi(x) v(t,x)   , \vspace{0.2cm} \\
			v(0,x)&=v_{0}(x) \in \Li^1_+\left(\R^2\right). 
		\end{array}	
		\right.
	\end{equation}

Then the system \eqref{3.32} has a unique non-negative mild solution. 
	
\end{corollary}
\begin{proof}
	It is sufficient to observe that the system \eqref{3.32} is equivalent to 
		\begin{equation*}  
		\left\{ 
		\begin{array}{rl}
			\partial_t v(t,x)&= \varepsilon^2 \bigtriangleup_x v(t,x) - \mathbf{\nabla}_x \cdot  \left(  v(t,x) \, \mathbf{C}(x) 	\right)   \vspace{0.2cm} \\
			&+ \left( \chi(x)+ \partial_{x_1}\mathbf{C}(x) _1  +\partial_{x_2}\mathbf{C}(x) _2  \right)v(t,x)   , \vspace{0.2cm} \\
			v(0,x)&=v_{0}(x) \in \Li^1_+\left(\R^2\right),
		\end{array}	
		\right.
	\end{equation*}
and the result follows from Theorem \ref{TH3.12}, and by using the variation of constant formula for $\lambda>0$ large enough
$$
v(t)= T_{A+B- \lambda I}(t)v_0+ \int_{0}^{t}T_{A+B- \lambda I}(t-\sigma) \left(L+ \lambda I\right)v\left(\sigma\right)\d \sigma, 
$$
where $L$ is the multiplicative operator 
$$
Lv(x)=  \left( \chi(x)+ \partial_{x_1}\mathbf{C}(x) _1  +\partial_{x_2}\mathbf{C}(x) _2  \right)v(x).
$$ 
\end{proof}

	We could obtain a stronger positivity result of $$	v(t,x)=T_{A+B_y}(t)(v_0)(x)$$ by using the strong maximum principle for a parabolic equation in the book of Gilbarg and Trudinger \cite{Gilbarg-Trudinger}. Alternatively, we could have used the Harnack inequality for second-order parabolic equations obtained by Ignatova, Kukavica, and Ryzhik	\cite{Ignatova-Kukavica-Ryzhik} to prove the strict positivity of the solution for all $t>0$ (by contradiction). But the simple arguments used above are sufficient to establish the convergence result of the entire system. 
\subsubsection{Conservation of the total mass of individuals}
Moreover by using again the formula \eqref{3.29}, we obtain 
\begin{equation*}
	\int_{\R^2}v_\lambda(t,x)dx =e^{-\lambda t}   	\int_{\R^2}v_0(x)dx+ \int_0^t  e^{-\lambda \left(t-s\right)}   \lambda 	\int_{\R^2}v_\lambda(s,x)dxds ,
\end{equation*}
that is 
\begin{equation*}
	\dfrac{d}{dt}	\int_{\R^2}v_\lambda(t,x)dx =-\lambda 	\int_{\R^2}v_\lambda(t,x)dx+ \lambda 	\int_{\R^2}v_\lambda(s,x)dx=0, 
\end{equation*}
therefore 
\begin{equation*}
	\int_{\R^2}v_\lambda(t,x)dx =	\int_{\R^2}v_0(x)dx, \forall t \geq 0. 
\end{equation*}
By using Lemma \ref{LE3.11}, we obtain the following theorem. 
\begin{theorem}[Conservation of the total mass of individuals] \label{TH3.14} Assume that $\mathbf{C_y}$ satisfies Assumption \ref{ASS3.3}.
	For each $v_0 \in \Li^1_+\left(\R^2\right)$, the Cauchy problem  \eqref{3.14}  conserves of the total mass of individuals. That is 
	\begin{equation*}
		\int_{\R^2}v(t,x)dx=	\int_{\R^2}v_0(x)dx, \forall t \geq 0, 
	\end{equation*}
	with  
	\begin{equation*}
		v(t)=T_{A+B_y}(t)v_0, \forall t \geq 0.		
	\end{equation*}
\end{theorem}

\section{Asymptotic behavior of the return-to-home model}
\label{Section4}
In this section, for simplicity, we drop the subscript $y$ notation, and we consider the entire system 
\begin{equation}
	\left\{ 
	\begin{array}{l}
		u'(t)=	\chi \int_{\R^2}w(t,x)\dx -\gamma \, u(t),\\
		v'(t,x)= \left(A_y+B_y \right)v(t,x)- \alpha \, v(t,x)+ \gamma \, g(x-y) \, u(t),\\
		w'(t,x)=\alpha\,  v(t,x)-\chi \, w(t,x),
	\end{array}
	\right. 
\end{equation}
with initial distribution 
\begin{equation}
	u(0)=u_0 \in \R_+, v(0)=v_0 \in L^1_+(\R^2), \text{ and } w(0)=w_0 \in L^1_+(\R^2). 
\end{equation}
\subsection{Abstract Cauchy problem}
We consider the space
$$
X=\R \times  \Li^1 \left(\R^2\right) \times \Li^1 \left(\R^2\right),
$$
which is a Banach space endowed with the standard produce norm 
$$
\Vert (u,v,w)\Vert= \vert u \vert + \Vert v \Vert_{ \Li^1\left(\R^2\right) }+\Vert w \Vert_{ \Li^1\left(\R^2\right) }. 
$$
We consider the positive cone of $X$ 
$$
X_+ =\R_+ \times  L^1_+ \left(\R^2\right) \times L^1_+ \left(\R^2\right).
$$
The system \eqref{2.1} we can rewritten as an abstract Cauchy problem 
\begin{equation} \label{4.1}
	\left(\begin{array}{c}
		u'(t)\\
		v'(t)\\
		w'(t)
	\end{array}\right)= \left( \mathcal{A}_y + \mathcal{C}_y  \right)	\left(\begin{array}{c}
		u(t)\\
		v(t)\\
		w(t)
	\end{array}\right), \text{ for } t >0,
\end{equation}
with initial value 
\begin{equation} \label{4.2}
	\left(\begin{array}{c}
		u(0)\\
		v(0)\\
		w(0)
	\end{array}\right)=	\left(\begin{array}{c}
		u_0\\
		v_0\\
		w_0
	\end{array}\right) \in \R \times  \Li^1 \left(\R^2\right) \times \Li^1 \left(\R^2\right).
\end{equation}
The linear operator $\mathcal{A}_y: D \left( \mathcal{A}_y \right)\subset X \to X $ is defined by 
$$
\mathcal{A}_y\left(\begin{array}{c}
	u \\
	v\\
	w
\end{array}\right)=\left(\begin{array}{c}
	-\gamma \, u \\
	\left(A+B_y  \right)v- \alpha \, v \\
	\alpha v  -\chi \, w
\end{array}\right)
$$
with the domain 
$$
D \left( \mathcal{A}_y \right)=\R \times D(A) \times \Li^1 \left(\R^2\right).
$$
The semigroup generated by $\mathcal{A}_y$ is explicitly given by 
\begin{equation} \label{4.5}
	T_{\mathcal{A}_y} (t) \left(\begin{array}{c}
		u_0 \\
		v_0\\
		w_0
	\end{array}\right)=\left(\begin{array}{c}
		e^{-\gamma t} \, u_0 \\
		T_{A+B_y -\alpha I}(t) v_0 \\
		e^{-\chi t}w_0+ \int_{0}^{t}e^{-\chi (t-s)}  \alpha  T_{A+B_y -\alpha I}(s) v_0 \d s 
	\end{array}\right). 	
\end{equation}
where 
$$
T_{A+B_y -\alpha I}(t)=e^{-\alpha t} T_{A+B_y }(t), \forall t \geq 0. 
$$
We also consider the compact bounded linear operator $\mathcal{C}_y: X  \to X,$ 
$$
\mathcal{C}_y\left(\begin{array}{c}
	u \\
	v\\
	w
\end{array}\right)=\left(\begin{array}{c}
	\chi \int_{\R^2}w(x)\dx \\
	\\
	+\gamma g(x-y) u    \vspace{0.2cm} \\
	0_{ \Li^1 \left(\R^2\right) } 
\end{array}\right).
$$

\begin{theorem}[Existence and uniqueness of solutions] \label{TH4.1} Assume that $\mathbf{C_y}$ satisfies Assumption \ref{ASS3.3}. Then for each $y \in \R^2$, the system \eqref{2.1} generates a strongly continuous semigroup $\left\{T_{\mathcal{A}_y+ \mathcal{C}_y}(t)\right\}_{t \geq 0}$  of bounded linear operator on $X$. We recall that 
	$$
	\left(\begin{array}{c}
		u(t) \\
		v(t)\\
		w(t)
	\end{array}\right)  =T_{\mathcal{A}_y+ \mathcal{C}_y}(t)\left(\begin{array}{c}
		u_0 \\
		v_0\\
		w_0
	\end{array}\right) 
	$$
	is the unique mild solution satisfies the variation of constant formula 
	\begin{equation}
		\left(\begin{array}{c}
			u(t) \\
			v(t)\\
			w(t)
		\end{array}\right) =T_{\mathcal{A}_y}(t)\left(\begin{array}{c}
			u_0 \\
			v_0\\
			w_0
		\end{array}\right) +  \int_{0}^{t} T_{\mathcal{A}_y}(t-\sigma)\,  \mathcal{C}_y\left(\begin{array}{c}
			u(\sigma) \\
			v(\sigma)\\
			w(\sigma)
		\end{array}\right)  \d \sigma, \forall t \geq 0, 
	\end{equation}
	or equivalently 
	\begin{equation}
		\left\{ 
		\begin{array}{rl}
			u(t)=&e^{-\gamma t} u_0+ \int_{0}^{t} e^{-\gamma \left(t- \sigma \right)} \chi \int_{\R^2}w(\sigma,x)\dx  \d \sigma, \\
			v(t)= &T_{A+B_y-\alpha I}(t)v_0+ \int_{0}^{t} T_{A+B_y-\alpha I}(t-\sigma) \gamma g(.-y) u(\sigma ) \d \sigma,\\
			w(t)=&	e^{-\chi t}w_0+ \int_{0}^{t}e^{-\chi (t-s)}  \alpha  v(t,\sigma) \d \sigma.  
		\end{array}
		\right. 
	\end{equation}
\end{theorem}

\noindent By Theorem \ref{TH3.12}, we have the following result. 
\begin{theorem}[Positivity] \label{TH4.2} The semigroup $\left\{T_{\mathcal{A}_y+ \mathcal{B}_y}(t)\right\}_{t \geq 0}$  is positive. That is 
	\begin{equation} \label{4.8}
		T_{\mathcal{A}_y+ \mathcal{B}_y}(t)  X_+  \subset X_+, \forall t \geq 0. 
	\end{equation}
\end{theorem}

\noindent By using Theorem \ref{TH3.14} we have the following result. 
\begin{theorem}[Conservation of the total mass] \label{TH4.3} Define 
	$$
	V(t)=\int_{\R^2} v(t,x)\d x, \text{ and   } W(t)=\int_{\R^2} w(t,x)\d x.
	$$
	Then $t \mapsto (u(t),V(t),W(t))$ satisfies the system of linear ordinary differential equations 
	\begin{equation} \label{4.9}
		\left\{ 
		\begin{array}{rl}
			u'(t)=&	\; \;\;  \chi W(t) -\gamma u(t),\\
			V'(t)= &- \alpha V(t)+ \gamma u(t),\\
			W'(t)=& \; \;\; \alpha V(t)-\chi W(t),
		\end{array}
		\right. 
	\end{equation}
	with initial distribution 
	\begin{equation} \label{4.10}
		u(0)=u_0, V(0)=\int_{\R^2} v_0(x) \d x, \text{ and } W(0)=\int_{\R^2} w_0(x) \d x.
	\end{equation}
	The density of individuals per house remains constant with time. That is 
	\begin{equation}\label{4.11}
		n(y)=	u(t)+ \int_{\R^2}v(t,x)\dx+  \int_{\R^2}w(t,x)\dx, \forall t \geq 0, \forall y \in \R^2, 	
	\end{equation}
	where $n(y)$ is the density of home in $\R^2$. 	
\end{theorem}

\begin{definition}\label{DE4.4}
	{\rm Let $T\in \mathcal{L}\left( X\right).$  Then the \textit{essential semi-norm} $\left\Vert T\right\Vert _{\rm ess}$ of $T$ is defined by
		\begin{equation*}
			\left\Vert T\right\Vert _{\rm ess}=\kappa \left( T\left( B_{X}(0,1)\right)\right), 
		\end{equation*}%
		where $B_{X}\left( 0,1\right) =\left\{ x\in X:\left\Vert
		x\right\Vert _{X}\leq 1\right\} ,$ and for each bounded set $%
		B\subset X$,
		\begin{equation*}
			\kappa \left( B\right) =\inf \left\{ \varepsilon >0:B\text{ can be covered
				by a finite number of balls of radius }\leq \varepsilon \right\}
		\end{equation*}%
		is the \textit{Kuratovsky measure of non-compactness}.}
\end{definition}

By using, Webb \cite{Webb79} (see Magal and Thieme \cite[Theorem 3.2.]{Magal-Thieme} for more results), we deduce that 
$$
x \mapsto  \int_{0}^{t} T_{\mathcal{A}_y}(t-\sigma)  \mathcal{C}_yT_{\mathcal{A}_y+ \mathcal{C}_y}(\sigma)x  \d \sigma, 
$$
is compact. Therefore, we obtain the following lemma. 

\medskip 
In  the following lemma, we are using  the essential growth rate of semigroup, we refer to Engel and Nagel \cite{Engel-Nagel}, or Magal and Ruan \cite{Magal-Ruan} for more results on this topic.  
\begin{lemma} \label{LE4.5} Assume that $\mathbf{C_y}$ satisfies Assumption \ref{ASS3.3}.  The essential growth rate 
	$$
	\omega_{ess} \left(\mathcal{A}_y + \mathcal{C}_y \right):=\lim_{t \to \infty} \Vert T_{\mathcal{A}_y + \mathcal{C}_y }(t) \left(B_X \left(0,1\right)\right)\Vert _{\rm ess} \leq -\min(\gamma,\alpha,\chi).
	$$
	
\end{lemma} 

Thanks to the negative essential growth rate and since the positive orbits are bounded, we deduce that the positive orbits are relatively compact (i.e., their closure is compact), and we obtain the following theorem.  
\begin{proposition}
	
	Assume that $\mathbf{C_y}$ satisfies Assumption \ref{ASS3.3}.	The omega-limit set of each trajectory is defined by 
	$$
	\omega\left(\begin{array}{c}
		u_0 \\
		v_0\\
		w_0
	\end{array}\right)    := \bigcap_{t \geq 0} \overline{\bigcup_{s \geq t} \left\{ T_{\mathcal{A}_y + \mathcal{C}_y }(s) \left(\begin{array}{c}
			u_0 \\
			v_0\\
			w_0
		\end{array}\right)  \right\}}
	$$ 
	in a non-empty compact subset of $X$ and is contained in 
	\begin{equation} \label{4.12}
		\left\{\left(u_0, v_0, w_0 \right)  \in X_+ : 	u+ \int_{\R^2}v(x)\dx+  \int_{\R^2}w(x)\dx =n(y) \right\}. 
	\end{equation}
\end{proposition}

\subsection{Equilibria}
An equilibrium solution of the model \eqref{2.1} will satisfy 
\begin{equation}\label{4.13}
	\left\{\begin{array}{ll}
		0=\chi \int_{\R^2}\overline{w}(x)\dx-\gamma \overline{u}(y),   \vspace{0.2cm} \\
		0=\varepsilon^2\Delta_{x} \overline{v}(x)- \mathbf{\nabla}_x \cdot  \left(  \overline{v}(x)\, \mathbf{C_y}(x)	 \right)-\alpha  \overline{v}(x)+\gamma g(x-y) \overline{u}(y),    \vspace{0.2cm} \\
		0=\alpha  \overline{v}-\chi  \overline{w}, 
	\end{array}\right.
\end{equation}
From the first and the equation of \eqref{4.13}, we deduce that 
\begin{equation} \label{4.14}
	\overline{w}_y(x)= \dfrac{\alpha}{\chi}  \overline{v}_y(x), \text{ and } \overline{u}(y)=\dfrac{\chi}{\gamma} \int_{\R^2}\overline{w}_y(x)\dx.
\end{equation}
By using the conservation of the total number of individuals in each house, we have   
$$
\overline{u}(y)+ \int_{\R^2}\overline{v}_y(y)\dx+  \int_{\R^2}\overline{w}_y(x)\dx=n(y),
$$
and by using \eqref{4.14}, we deduce that 
\begin{equation} \label{4.15}
	\overline{u}(y)=\dfrac{\tau}{\gamma} n(y)=\dfrac{1}{ \left(1+ \dfrac{\gamma}{\alpha}+\dfrac{\gamma}{\chi} \right)}  n(y),
\end{equation}
where 
$$
\tau=\dfrac{1}{ \left(\dfrac{1}{\gamma}+ \dfrac{1}{\alpha}+\dfrac{1}{\chi} \right)}.
$$
By plugging \eqref{4.15} into the $v$-equation of \eqref{4.13}, we deduce that 
$$
0=\varepsilon^2\Delta_{x} \overline{v}_y-\mathbf{\nabla}_x \cdot  \left(  \overline{v}_y\, \mathbf{C}_y	 \right)-\alpha  \overline{v}_y+\tau g(x-y) n(y),
$$
which is equivalent 
$$
\alpha  \overline{v}_y-	\varepsilon^2\Delta_{x} \overline{v}+\mathbf{\nabla}_x \cdot  \left(  \overline{v}_y\, \mathbf{C}_y \right)=\tau g(x-y) n(y).
$$
Therefore 
\begin{equation*} 
	\overline{v}_y(x)= \bigg(\alpha I-A-B_y	\bigg)^{-1}\bigg(\tau g(\cdot-y) n(y)\bigg),
\end{equation*}
or equivalently 
\begin{equation} \label{4.16}
	\overline{v}_y(x)= \tau n(y) \int_{0}^{+\infty} e^{-\alpha t} T_{A+B_y}(t) \left(g(.-y)\right)(x)\d t.
\end{equation}

\subsection{Asymptotic behavior}

By integrating in $x$ $v(t,x)$ and $w(t,x)$ and by using  
\begin{equation}
	\left\{ 
	\begin{array}{l}
		u'(t)=	\chi \int_{\R^2}w(t,x)\dx -\gamma u(t),\\
		v'(t)= \left(A_y+B_y \right)v(t)- \alpha v(t)+ \gamma g(.-y) u(t),\\
		w'(t)=\alpha v(t)-\chi w(t),
	\end{array}
	\right. 
\end{equation}
with initial distribution 
\begin{equation}
	u(0)=U_0 \in \R_+, v(0)=v_0 \in L^1_+(\R^2), \text{ and } w(0)=w_0 \in L^1_+(\R^2). 
\end{equation}

By using Perron-Frobenius theorem applied to the irreducible system  \eqref{4.9} we obtain the following theorem. We refer to Ducrot, Griette, Liu, and Magal \cite[Theorem 4.53]{DGLM-2022} for more result on this subject. 
\begin{lemma} \label{LE4.7}
	Assume that $\alpha>0$ $\gamma>0$ and $\chi>0$. Then the solution of system \eqref{4.9} satisfies 
	\begin{equation} \label{4.19}
		\lim_{t \to \infty }u(t)=\overline{u}, \lim_{t \to \infty }V(t)=\overline{V},  \text{ and } \lim_{t \to \infty }W(t)=\overline{W},
	\end{equation}
	where
	\begin{equation*}
		\overline{u}\left(1+\dfrac{\gamma}{\alpha}+ \dfrac{\gamma}{\chi}\right)=n(y), 
	\end{equation*} 
	\begin{equation*}
		\overline{V}\left(\dfrac{\alpha}{\gamma}+1+ \dfrac{\alpha}{\chi}\right)=n(y), 
	\end{equation*}
	and 
	\begin{equation*}
		\overline{W}\left(\dfrac{\chi}{\gamma}+ \dfrac{\chi}{\alpha}+1\right)=n(y).
	\end{equation*}
	Moreover, the convergence in \eqref{4.19} is exponential. That is, there exists a constant $M>0$ and $\delta>0$ such that for each $t \geq 0$,  
	\begin{equation} \label{4.20}
		\vert u(t)-\overline{u}\vert \leq M e^{-\delta t}, 	\vert V(t)-\overline{V}\vert \leq M e^{-\delta t},  \text{ and }	\vert W(t)-\overline{W}\vert \leq M e^{-\delta t}. 
	\end{equation}
\end{lemma}
\begin{proof} The matrix of system \eqref{4.9} is 
	$$
	L= 	\begin{pNiceMatrix}
		-\gamma     &  0 &\chi     \\
		\gamma  & -\alpha     & 0  \\
		0  &    \alpha      & -\chi \\ 
	\end{pNiceMatrix} .
	$$
	Therefore the system \eqref{4.9} is strongly connected (i.e., $L+\delta I$ is irreducible for all $\delta>0$ large enough). The vector $\mathbbm{1}^T=(1,1,1)^T$ is a strictly positive left-eigenvector associated with the eigenvalue $0$. The Perron-Frobenius theorem shows that $0$ is the dominant eigenvalue of $L$ (i.e., an eigenvalue with the largest real part).   The equilibrium of equation \eqref{4.9} corresponds to the right eigenvector. That is 
	\begin{equation} \label{4.21}
		\chi \overline{W} =\gamma \overline{u}, \alpha \overline{V}=  \gamma \overline{u} , \alpha \overline{V} =\chi \overline{W} ,
	\end{equation}
	and since we must impose that 
	$$
	\overline{u}+\overline{V}+\overline{W}=n(y),
	$$
	the proof is completed. 
\end{proof} 

\begin{theorem} \label{TH4.8} Assume that $\mathbf{C_y}$ satisfies Assumption \ref{ASS3.3}. Assume that $\alpha>0$ $\gamma>0$ and $\chi>0$.
	For each $y \in \R^2$, the solution of system \eqref{2.1} satisfies 
	\begin{equation} \label{4.22}
		\lim_{t \to+\infty} u_y(t)=  \overline{u}(y), \text{ in }  \R, 
	\end{equation}    
	
	\begin{equation} \label{4.23}
		\lim_{t \to+\infty} v_y(t,x)=  \overline{v}_y(x), \text{ in }  \Li^1(\R^2), 	
	\end{equation}       
	and 
	\begin{equation}	\label{4.24}
		\lim_{t \to+\infty} w_y(t,x)=  \overline{w}_y(x), \text{ in }  \Li^1(\R^2),
	\end{equation}
	and the convergence is exponential for each limit. 
	
\end{theorem}
\begin{proof}
	By Lemma \ref{LE4.7}, we already know the exponential convergence in \eqref{4.22}. Let us consider the exponential convergence in  \eqref{4.23}. We have 
	$$
	\begin{array}{rl}
		v(t,x)&=T_{A_y+B_y - \alpha I}(t)v_0+\int_{0}^{t}  T_{A_y+B_y - \alpha I}(t-\sigma) \gamma g(.-y) u(\sigma) \d \sigma,
	\end{array}
	$$ 
	and 
	$$
	\begin{array}{rl}
		\overline{v}(x)&=T_{A_y+B_y - \alpha I}(t)\overline{v}+\int_{0}^{t}  T_{A_y+B_y - \alpha I}(t-\sigma) \gamma g(.-y) \overline{u}(\sigma) \d \sigma.
	\end{array}
	$$ 
	Therefore, we deduce that  
	$$
	\begin{array}{rl}
		v(t,x)-\overline{v}(x)&=T_{A_y+B_y - \alpha I}(t) \left(v_0-\overline{v}\right)\\
		&+\int_{0}^{t}  T_{A_y+B_y - \alpha I}(t-\sigma) \gamma g(.-y) \left(u(\sigma) -\overline{u}\right)\d \sigma,
	\end{array}
	$$ 
	and we obtain 
	$$
	\begin{array}{rl}
		\Vert v(t)-\overline{v} \Vert_{ \Li^1\left(\R^2\right) } \leq &e^{-\alpha t} \Vert v_0-\overline{v} \Vert_{ \Li^1\left(\R^2\right) } \\
		&+ \int_{0}^{t} e^{-\alpha \left(t-\sigma\right)}\vert u(\sigma)-\overline{u} \vert \d\sigma.
	\end{array}
	$$
	Now, by Lemma \ref{LE4.7}, we have $\vert u(t)-\overline{u}\vert \leq M e^{-\delta t}, \forall t \geq 0$, we obtain 
	$$
	\begin{array}{rl}
		\Vert v(t)-\overline{v} \Vert_{ \Li^1\left(\R^2\right) } \leq &e^{-\eta t }(1+t) \left( \Vert v_0-\overline{v} \Vert_{ \Li^1\left(\R^2\right) }+M\right) \\
	\end{array}
	$$
	where $\eta= \min \left(\alpha, \delta \right)>0$.  The exponential convergence in  \eqref{4.24} follows by using the exponential convergence in \eqref{4.23} and using similar arguments to those above. 
\end{proof}
\begin{remark}
	The above result is relate the irreducibility of the semigroup $\left\{T_{\mathcal{A}_y+ \mathcal{C}_y}(t)\right\}_{t \geq 0}$. The difficulty would be to prove the additional result 
	for each $\phi \in  \Li^\infty_+\left(\R^2\right)$,  with $\phi  \neq 0$,  we have 
	$$
	\int_{\R^2} \phi(x) u(t,x)dx >0, \forall t >0,
	$$
	where  
	\begin{equation}
		u(t,x)=T_{A+B_y}(t)(g(.-y))(x), \forall t \geq 0.		
	\end{equation}
	The reader can find more result on this topic in the paper by  Webb \cite[Remark 2.2]{Webb87} (see also  \cite{Arendt2020, Arendt, Arino, Grabosch, Greiner} for more on this subject) to prove  infinite dimensional Perron-Frobenius like theorem.  Here, we propose a more direct approach to study the asymptotic behavior of the system.
\end{remark}

\section{Hybrid formulation of a return to home model}
\label{Section5}
A major difficulty in applying such a model in concrete situations is the computation time. Indeed the time of computation grows exponentially with the discretization step. In the previous section, we introduced reduction technique, that could be used to run the simulations of the return home model. Unfortunately such an idea does not apply to the case of epidemic model.  To circumvent this difficulty, we now introduce discrete homes locations.

\begin{assumption} \label{ASS5.1}
	Assume that we can find a sequence of point $ y_i =\left(y^i_1,y^i_2 \right) \in \R^2 $, and the index $i$ belongs to a countable set $I$. 
\end{assumption}

\begin{remark} \label{ASS5.2} In the numerical simulations section, it will be convenient to use a finite number of  homes 
	$$        
	I=\left\{1, \ldots,n\right\}. 
	$$        
	But, we could also  consider a one dimensional lattice with $I=\Z$ or a two dimensional lattice with $I= \Z\times \Z$.  
\end{remark}

The model we consider now is the previous model in which we assume that  
$$
n(y)=u(t,y)+ \int_{\R}\left(v+w\right)(t,x,y)dy=\sum_{i \in I} n_i \delta_{y_i}(y)
$$ 
where $y \to  \delta_{y_i}(y)$ is the Dirac mass at $y_i$. 

\medskip 
Instead of considering 
$$
u(t,y)=\sum_{i \in I} u_i(t) \delta_{y_i}(y),
$$ 
it is sufficient to consider $\left(u_1 (t), \ldots, u_n (t) \right)\in \R^n $ the numbers of individual staying at home with their home located at $y_1, \ldots, y_n$. 

\medskip 
We define $v_i(t,x)$ (respectively $w_i(t,x)$) the density of travelers (respectively workers) with their home located at $y_i$, and $x \mapsto \mathbf{C}_i(x)$ the traveling speed of individual coming from the home located at $y_i$. 

\medskip 
The return home model consists of a decoupled system of $n$ sub-system of the following form  
\begin{equation}\label{5.1}
	\left\{\begin{array}{ll}
		\partial_{t} u_i(t)=\chi \int_{\R^2}w_i(t,x)\dx-\gamma u_i(t),   \vspace{0.2cm} \\
		\partial_{t} v_i(t,x)=\varepsilon^2\Delta_{x}v_i-\mathbf{\nabla}_x \cdot  \left( v_i\, \mathbf{C}_i(x)	 \right)-\alpha v_i+\gamma g(x-y_i)u_i(t),    \vspace{0.2cm} \\
		\partial_{t} w_i(t,x)=\alpha v_i(t,x)-\chi w_i(t,x), 
	\end{array}\right.
\end{equation}
with $i=1, \ldots, n$, $ t\geq 0$, $x \in \R^2$ is the spatial location individual, and $y_i \in \R^2,$ is their home's location,  and the initial distribution at $t=0,$ and for $i=1, \ldots, n,$ 
\begin{equation} \label{5.2}
	\left\{
	\begin{array}{l}
		u_i(0)=u_{i0} \in [0, +\infty)  , \vspace{0.2cm} \\
		v_i(0,x)=v_{i0}(x) \in \Li^1_+\left(\R^2\right), \vspace{0.2cm} \\
		\text{and } \vspace{0.2cm} \\
		w_i(0,x)=w_{i0}(x) \in \Li^1_+\left(\R^2\right).
	\end{array}	
	\right. 
\end{equation}

\medskip 
\noindent \textbf{Conservation of individuals: }Total number of individual in each house $i \in I$ is preserved 
$$
n_i= \underset{\text{
		\begin{tabular}{@{}c@{}}  
			Number of\\individuals\\ at home
\end{tabular}}}{\underbrace{  u_i(t)}}+\underset{\text{
		\begin{tabular}{@{}c@{}}  
			Number of\\ travelers
\end{tabular}}}{\underbrace{  \int_{\R^2}v_i(t,x)\dx}}+\underset{\text{
		\begin{tabular}{@{}c@{}}  
			Number of\\ workers
\end{tabular}}}{\underbrace{   \int_{\R^2}w_i(t,x)\dx}},
$$
is the number of individuals in the home $i$ at time $t$.  

\medskip 
\noindent \textbf{Equilibria:} For each $i \in I$, we have a unique equilibrium 
\begin{equation*} 
	\overline{u}_i=\dfrac{1}{ \left(1+ \dfrac{\gamma}{\alpha}+\dfrac{\gamma}{\chi} \right)}  n_i,
\end{equation*}
\begin{equation*} 
	\overline{v}_i(x)= \bigg(\alpha I-A-B_{y_i}	\bigg)^{-1}\bigg(\tau g(\cdot-y_i ) n_i\bigg),
\end{equation*}
and 
\begin{equation*}
	\overline{w}_i(x)= \dfrac{\alpha}{\chi}  \overline{v}_i(x).
\end{equation*}
As a consequence of Theorem \ref{TH4.8}.
\begin{corollary} \label{CO5.3} Assume that $\mathbf{C}_i$ satisfies Assumption \ref{ASS3.3}. Assume that $\alpha>0$ $\gamma>0$ and $\chi>0$.
	For each $i \in I$, the solution of system \eqref{5.1} satisfies 
	$$        
	\lim_{t \to+\infty} u_i(t)=  \overline{u}_i, \text{ in }  \R, 
	$$        
	$$        
	\lim_{t \to+\infty} v_i(t,x)=  \overline{v}_i(x), \text{ in }  \Li^1(\R^2), 
	$$        
	and       
	$$       
	\lim_{t \to+\infty} w_i(t,x)=  \overline{w}_i(x), \text{ in }  \Li^1(\R^2),
	$$        
	
\end{corollary}

\section{Numerical simulations of the hybrid model}
\label{Section6}
In this section, we run a simulations of the model \eqref{2.1} on a bounded domain 
$$
\Omega=[0,1] \times [0,1].
$$ 
The model on bounded is presented in Appendix \ref{SectionA}. Here, we use the following initial distribution   
$$
u_i(0)=n_i, v_i(0,x)=w_i(0,x)=0, \forall i \in I. 
$$
We assume that the convection is null. That is,   
$$
\mathbf{C}_i(x) =0. \forall x \in \Omega, \forall i \in I.
$$
It is essential to mention that the numerical results are obtained by using an Euler integration method 
$\Delta x_1 \Delta x_2 \sum_{j} \sum_{k} w_i(t,x_1^j,x_2^k )$ for $\int_{\Omega}w_i(t,x)dx$ in the $u$-equation of system \eqref{5.1}. This method does not give a very good approximation of the integral, but this approximation is preserved through the numerical scheme used for diffusion. For example, the Simpson method does not work to compute the solution, and the errors accumulate and produces a blowup of the solutions. 
In the numerical simulations, we use a semi-implicit numerical method to compute the diffusive part of the system (see Appendix \ref{SectionA}).

In Table \ref{Table1}, we list the parameters used in the simulations. 

\begin{table}[H]
	\begin{center}
		\begin{tabular}{| c | c | c | c | }
			\hline
			
			{\bf	  Symbol	}& 		{\bf Interpretation  } &{\bf Value }& 	 {\bf Unit  }  \\
			
			\hline 
			$	\varepsilon $&  Diffusion coefficient  & $1$ & none \\
			\hline 	  
			$1/ \gamma$	 & Average time spent at home & $12/24$ & day \\
			\hline 	  
			$1/ \alpha$	 & Average time spent traveling & $2/24$ &  day \\
			\hline 	  
			$1/ \chi$	 & Average time spent at work & $10/24$ &  day \\
			\hline 	  
			$\sigma$ & Standard deviation  for the function $\rho(x1,x2)$& $0.05$ & none \\
			\hline
		\end{tabular}
	\end{center}
	\caption{\textit{List of parameters used in the simulations. }}\label{Table1}
\end{table}
\noindent In Figure \ref{Fig4}, we plot the number of people per home with the location of each home. 
\begin{figure}[H]
	\begin{center}
		\includegraphics[scale=0.25]{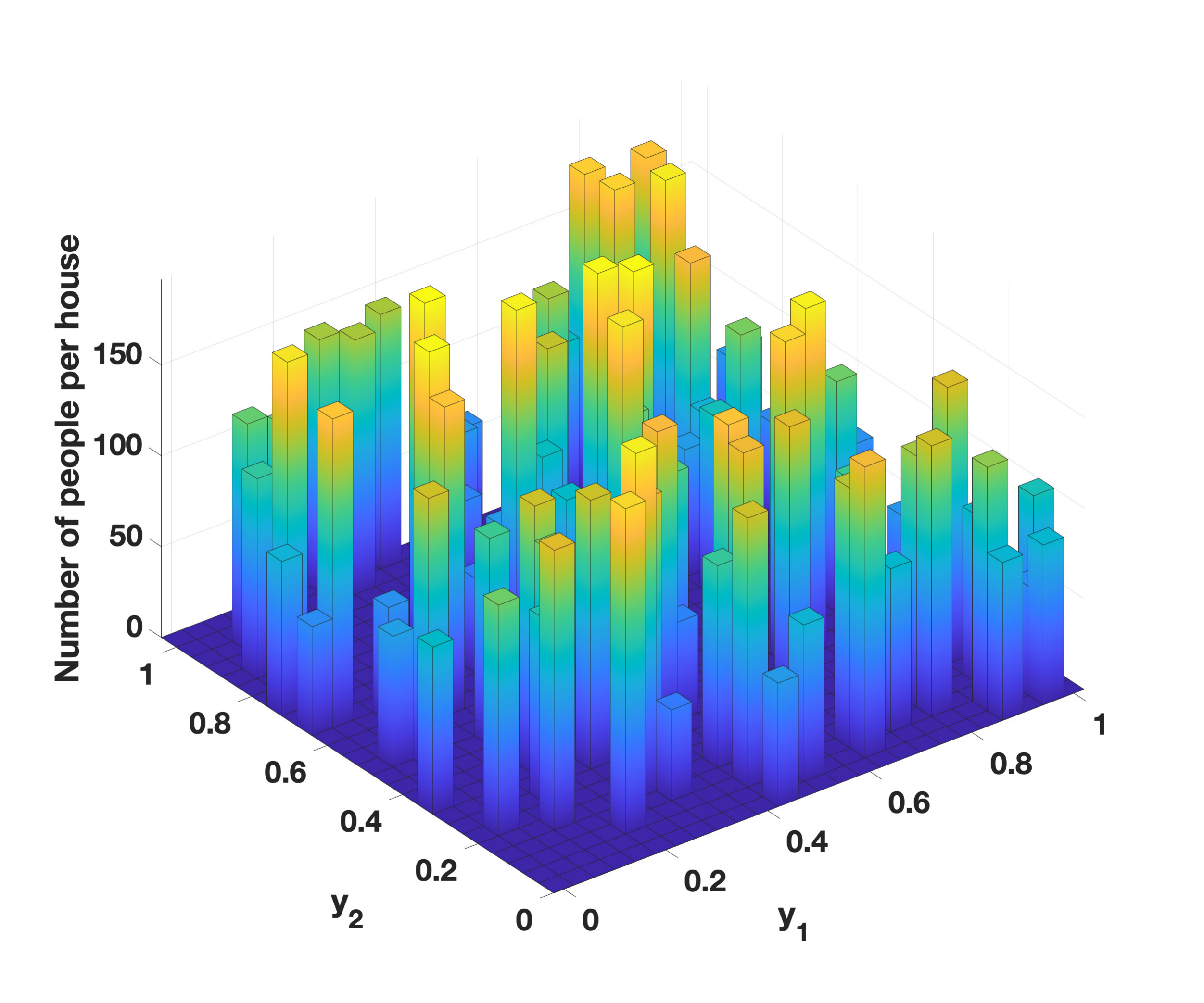}  
	\end{center}
	\caption{\textit{In this figure, we plot bars with height $n_i$  located at $y_i$, the number of individuals with their home located at $y_i \in \Omega$. The number of individuals per home varies randomly between $50$ and $200$ per home.}}\label{Fig4}
\end{figure}

In Figure \ref{Fig5},  we plot  
$$
x \mapsto \sum_{i \in I } n_i g(x-y_i)
$$ 
which is the density of individuals leaving their homes at time $t=0$. This figure gives another representation of the density of individuals at home.
\begin{figure}[H]
	\begin{center}
		\includegraphics[scale=0.18]{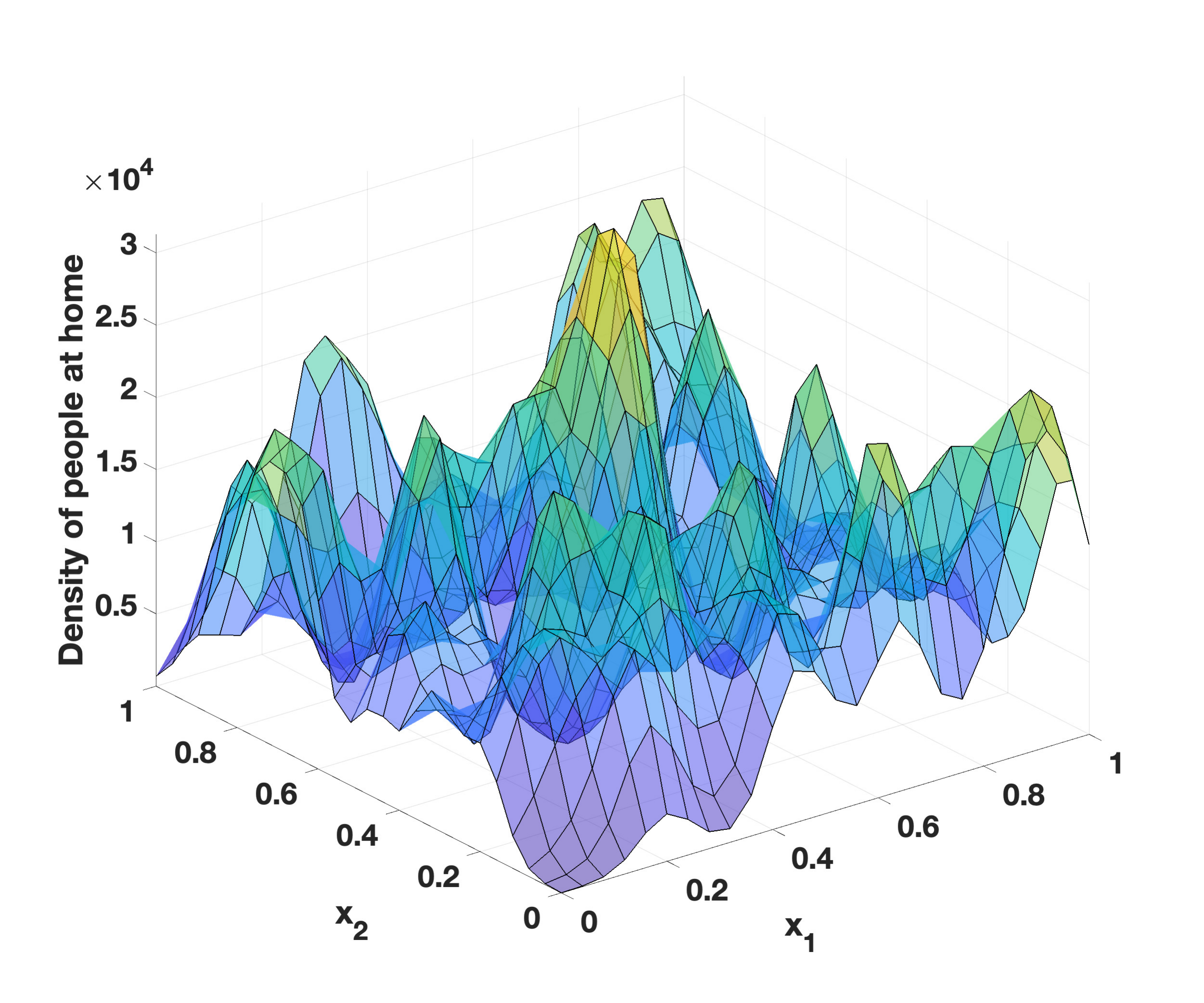}  
	\end{center}
	\caption{\textit{In this figure, we plot the total density of people at home $x \mapsto \sum_{i \in I } n_i g(x-y_i)$, where $N_i$ is the number of individuals in the home $i$, and $y_i$ is the location of the home $i$. This density of population  represents the distribution of individuals leaving their homes.  }}\label{Fig5}
\end{figure}

In Figure \ref{Fig6}, we observe that the numerical method preserves the number of individuals. 
\begin{figure}[H]
	\begin{center}
		\includegraphics[scale=0.18]{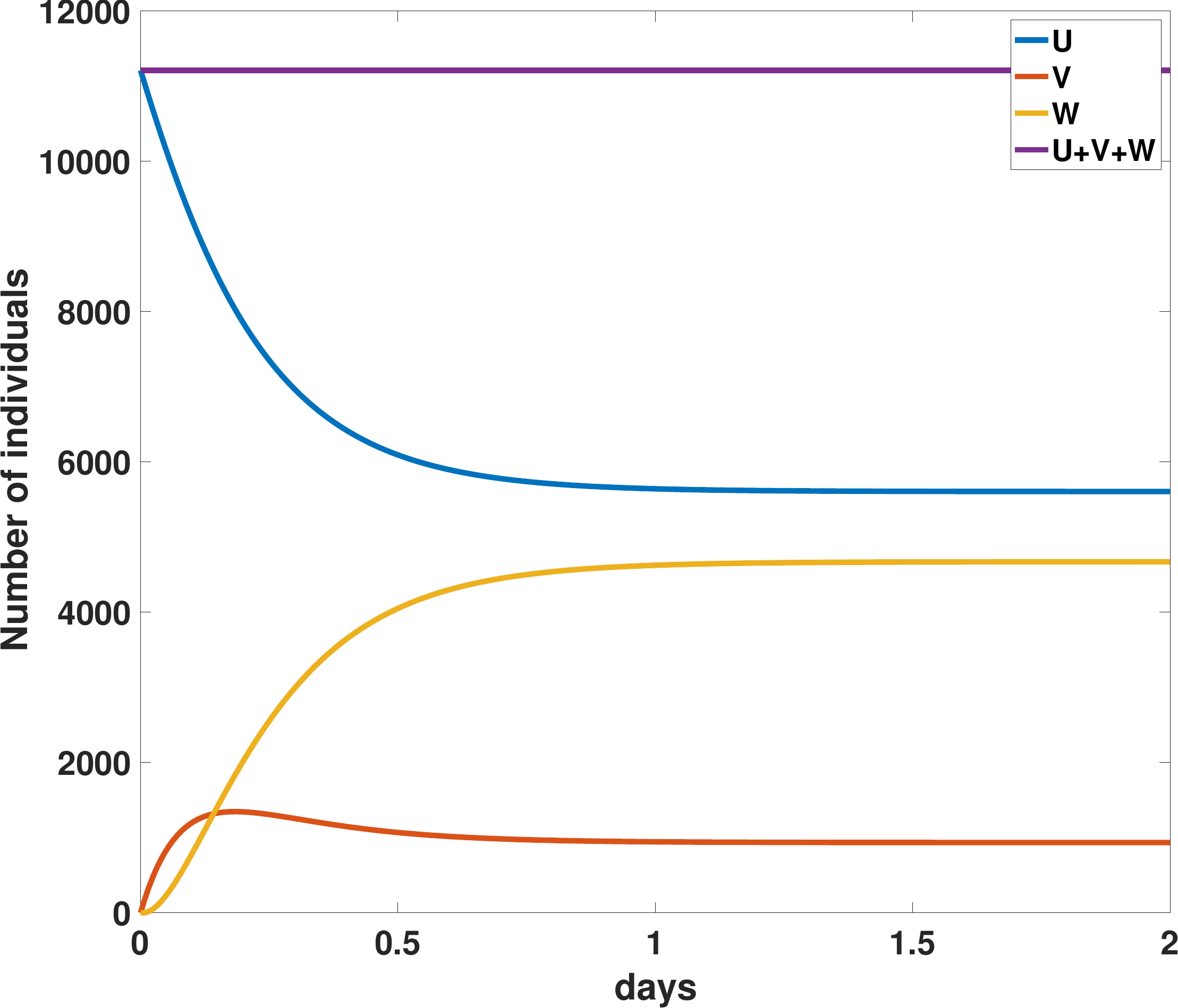}  
	\end{center}
	\caption{\textit{We plot the total number of individuals at home $t \to U(t)=\sum_{i \in I} u_i(t)$ (blue), the total number of travelers  $t \to V(t)=\sum_{i \in I} \int_{\Omega} v_i(t,x)dx$ (orange), and the total number of workers $t \to W(t)=\sum_{i \in I} \int_{\Omega} w_i(t,x)dx$ (yellow), the total number of individuals (purple). After one day, we observe the number of individuals in each compartment remains constant. }}\label{Fig6}
\end{figure}
In Figure \ref{Fig7}, we plot the number of people at home, travelers, and workers in each home at time $t=2$. That is 
$$
u_i(2), \,\,  \int_{\Omega} v_i(2,x)dx, \text{ and }  \int_{\Omega} w_i(2,x)dx, 
$$
and we draw a bar at their home location $y_i$. 

\medskip 
We observe that each distribution (a) (b) or (c) is a multiple of the density of individual at home $(y_1, y_2) \mapsto n \left(y_1, y_2\right)$,  and the individual are mixed subdivide in between each compartments. The maximal value is $180$  in (a), $15$  in (b), and  $70$ in (c).
\begin{figure}[H]
	\begin{center}
		\begin{minipage}{0.48\textwidth}
			\centering
			\includegraphics[width=\textwidth]{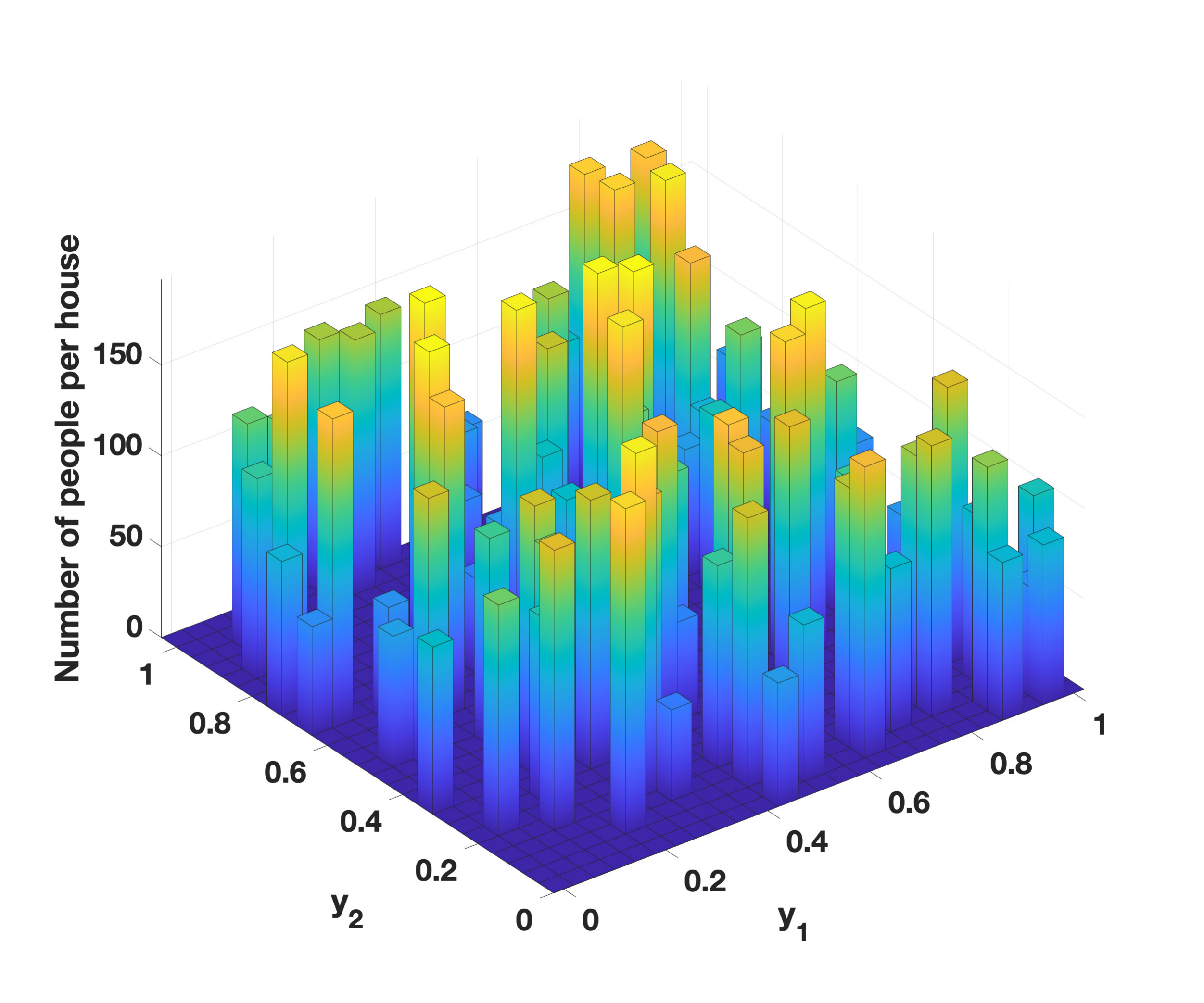}
		\end{minipage}
		\begin{minipage}{0.48\textwidth}
			\centering
			\includegraphics[width=\textwidth]{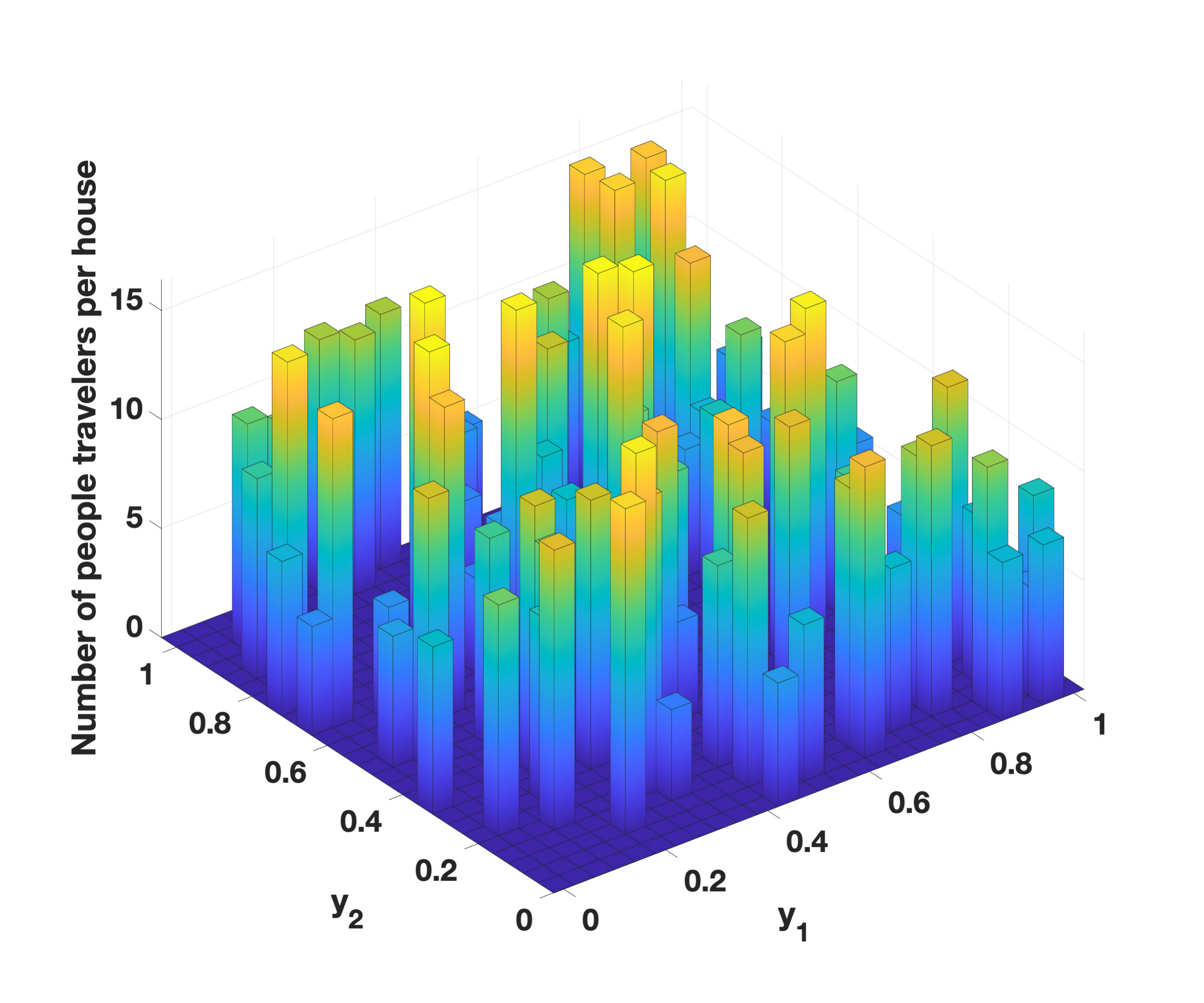}
		\end{minipage}
		\begin{minipage}{0.48\textwidth}
			\centering
			\includegraphics[width=\textwidth]{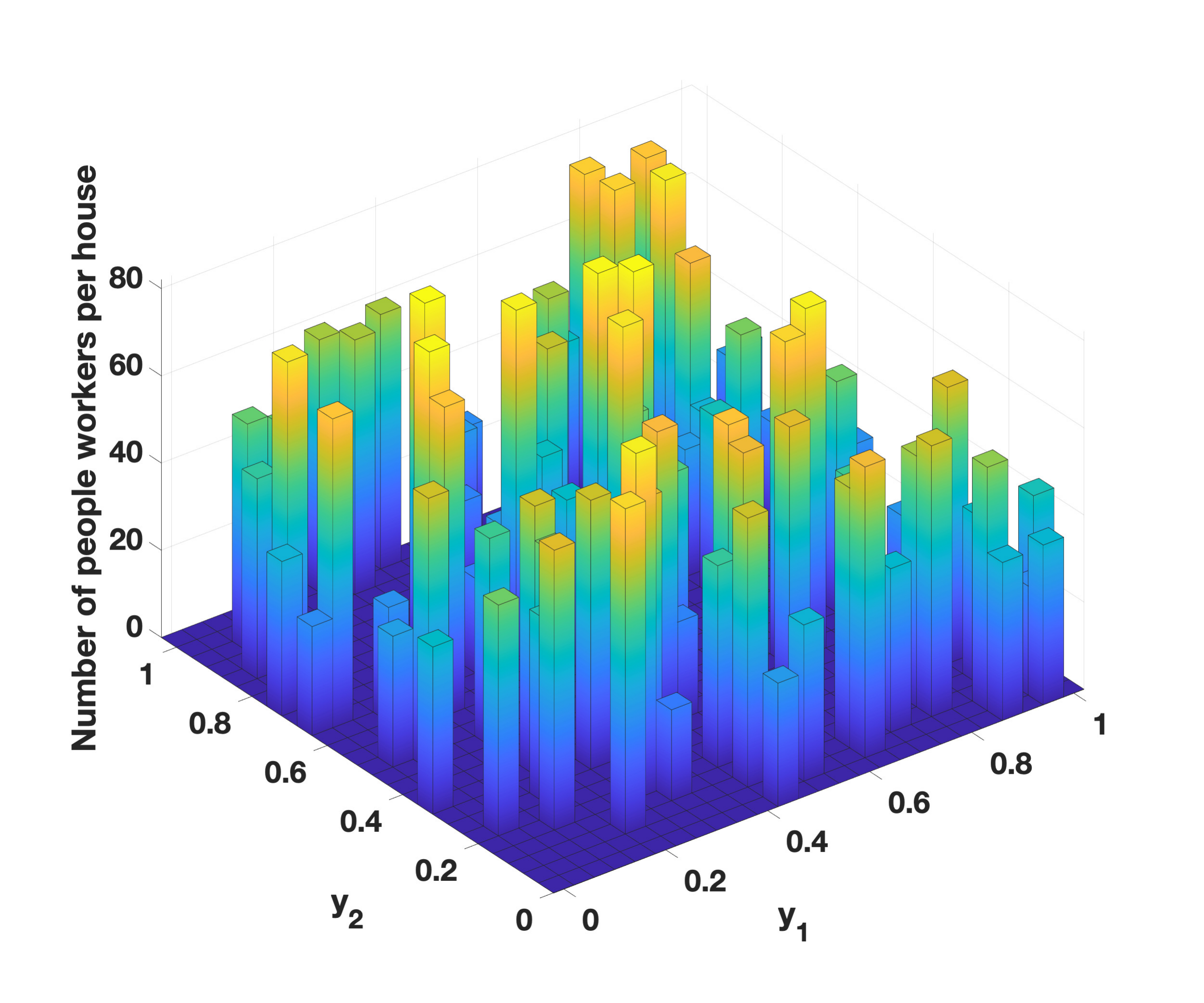}
		\end{minipage}
	\end{center}
	\caption{\textit{In this figure, we plot  the number of individuals on day $2$:  1) at home $y_i \mapsto u_i\left(2\right)$ (on the top left) for each home $i$; 2) traveling $y_i \mapsto \int_{\Omega} v_i(2,x)dx$ (on the top right) for each home $i$;  3) at work $y_i \mapsto \int_{\Omega} w_i(2,x)dx$ (on the bottom) for each home $i$.  The three figures look the same,  but their amplitude is very different. The maximal value is $150$ on the top left, $15$ on the top right, and $80$ on the bottom. } }\label{Fig7}
\end{figure}
In Figure \ref{Fig8}, we plot 
$$
\sum_{i \in I}  v_i(2,x_1,x_2), \text{ and }\sum_{i \in I}  w_i(2,x_1,x_2). 
$$
We observe numerically the equilibrium formula   \eqref{4.14}. That is, 
$$ 
\sum_{i \in I}  \overline{w}_i(2, x_1,x_2)= \dfrac{\alpha}{\chi}    \sum_{i \in I}  \overline{v}_i (2,x_1,x_2), \forall (x_1,x_2) \in \Omega.
$$ 
\begin{figure}[H]
	\begin{center}
		\begin{minipage}{0.48\textwidth}
			\centering
			\includegraphics[width=\textwidth]{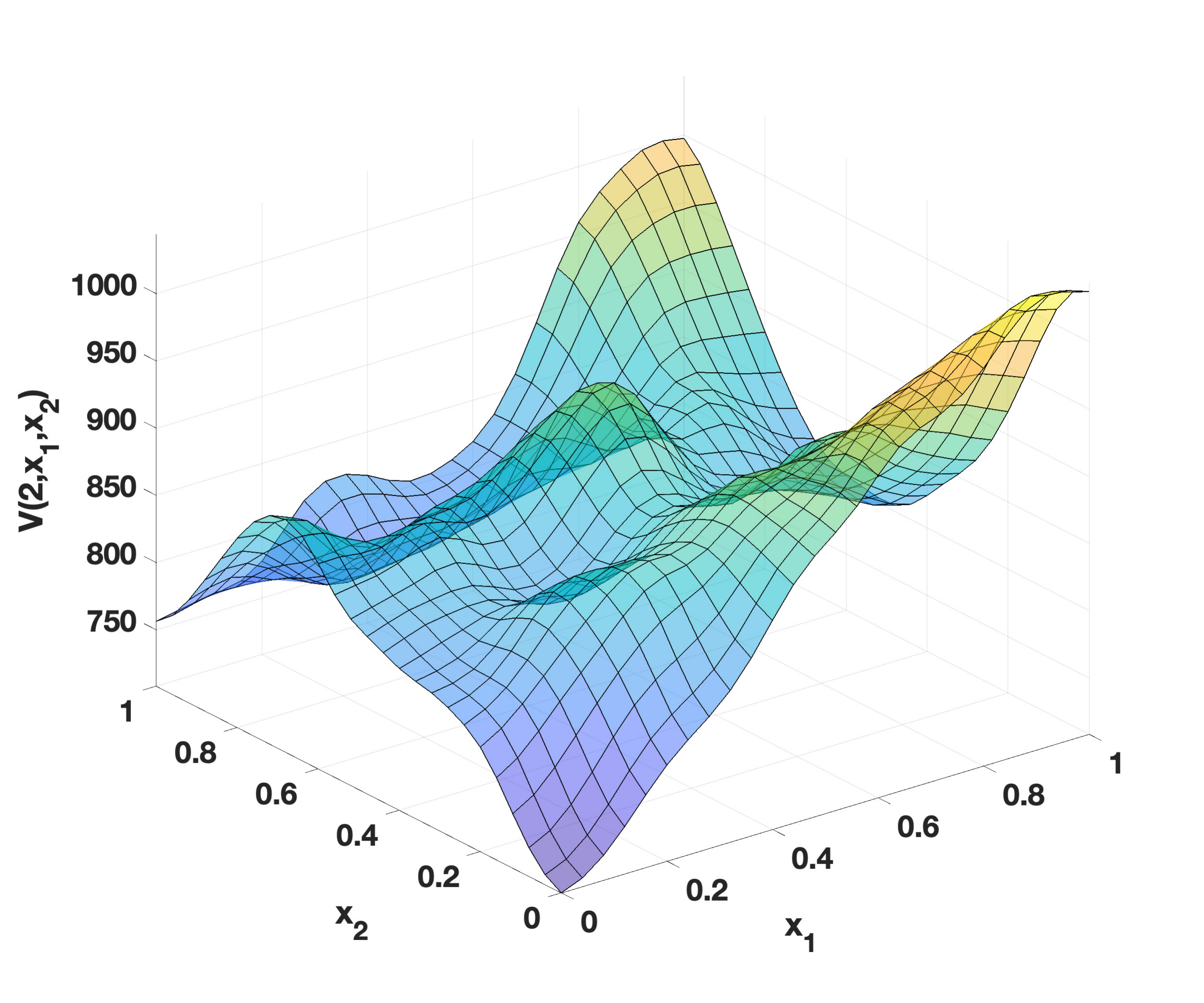}
		\end{minipage}
		\begin{minipage}{0.48\textwidth}
			\centering 
			\includegraphics[width=\textwidth]{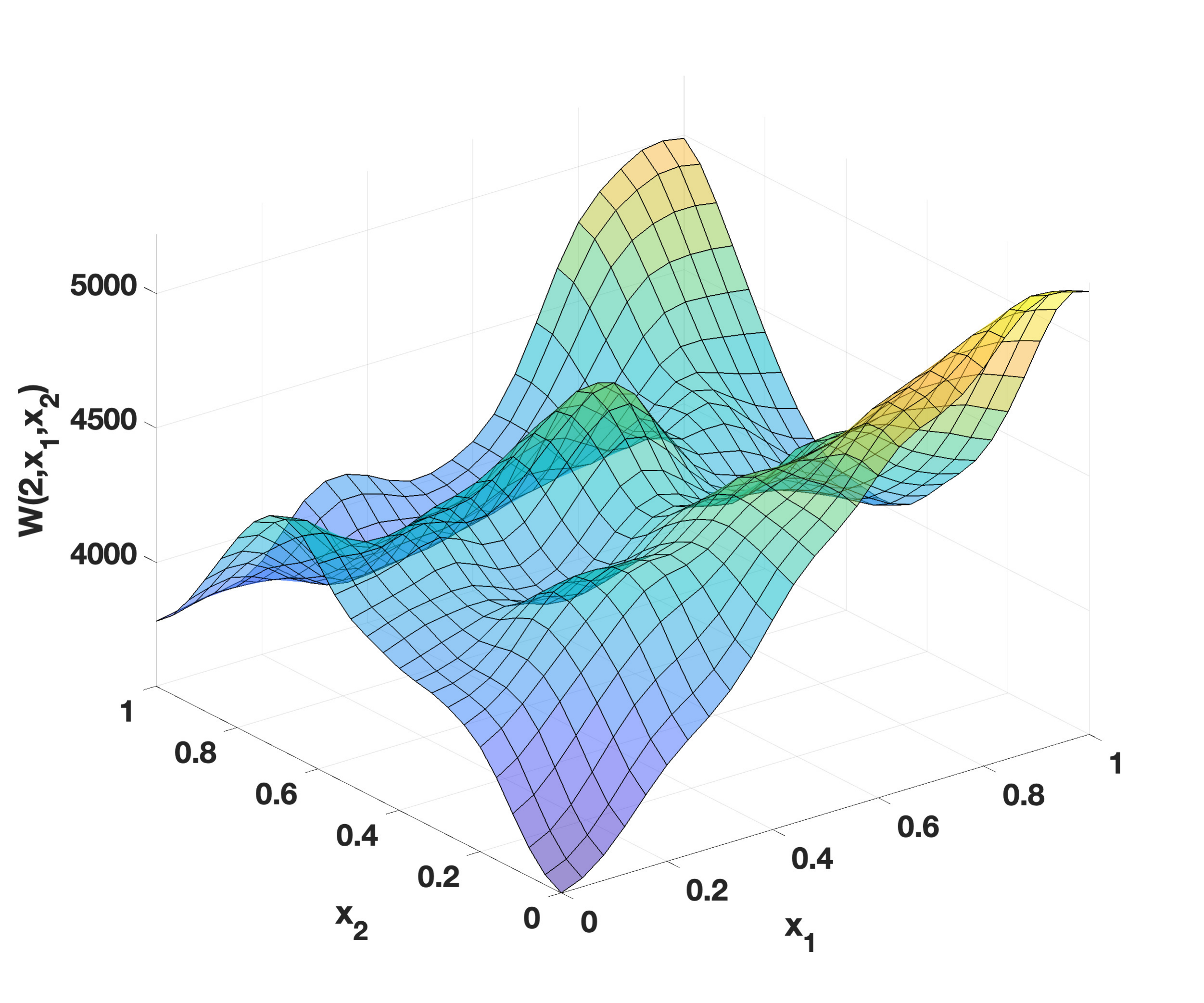}
		\end{minipage}
	\end{center}
	\caption{\textit{In the figure, we plot the distribution of all the travelers  $(x_1,x_2) \mapsto \sum_{i \in I}  v_i(2,x_1,x_2)$ (on the top), and  we plot the distribution of all the workers $(x_1,x_2) \mapsto \sum_{i \in I}  w_i(2,x_1,x_2)$ (on the bottom). Both figures, left and right, look the same, and only the amplitude changes from the left to the right.}}\label{Fig8}
\end{figure}

\section{Conclusion}
\label{Section7}
This article presents a new model, including a compartment for people at home, traveling, and people at work. We study the model's well-posedness and obtain a convergence result to a stationary distribution. The numerical simulations have illustrated such convergence results, and we observed that only one day is necessary for the solutions of the model to converge to the equilibrium distributions. Such a model is essential because significant social differences exist between individuals depending on their home location. Intuitively, the people living in the city's center would travel a short distance to work, while those living in the suburbs would travel a long distance to their working places. 

\medskip 
The model could be complexified in many ways. We could introduce multiple groups to describe the different types of behavior for people at work. For example, some people, like taxi drivers, never stop to travel while they are working.   Conversely, teleworking people stay at home to work but leave their homes to shop. 

\medskip 
We could also consider multiple transport speeds $\mathbf{C_y^k}(x)$ for people leaving their homes at $y \in \R^2$. Different speeds can match different means of transportation, car, bus, subway, etc. Assuming, for example, that $m$-types of transport speed are involve, then for each group  $k=1, \ldots, m$, we would have the following model to describe the travelers 

\begin{equation}  \label{7.1}
	\left\{ 
	\begin{array}{rl}
		\partial_t v_k(t,x)&= \varepsilon^2 \bigtriangleup_x v_k(t,x) - \mathbf{\nabla}_x \cdot  \left(  v_k(t,x) \, \mathbf{C_y^k}(x) 	\right), \vspace{0.2cm} \\
		v_k(0,x)&=v_{0}^k (x) \in \Li^1\left(\R^2\right). 
	\end{array}	
	\right.
\end{equation}

\medskip 
Suppose we consider now an epidemic spreading in a city. In that case, the most critical compartments are those staying at home and work, where most pathogens' transmissions occur. The return-to-home model could compute the distribution of people at work and home depending on their home locations in a given city. Return-to-home models could be used to study various phenomena in the cities. We can extend this model to study air pollution, the spread of epidemics, and other important problems to understand the population dynamics at the level of a single city.  

\medskip 	
Here we use a model to describe travelers' movement, which is relatively simplistic. For example, people travel on the roads, not through the buildings. Another question would be how to include the streets or a map in such a model. 

\medskip 	
To conclude the paper, we should mention that animals also have a home. An important example is the bee, and we refer to \cite{MWW1, MWW2} for more results on this topic. Many species of animals live around their home, so modeling return-to-home is probably essential to understand the dynamics of many leaving populations. 

\medskip 	
This article considers the case where the model's parameters are constant with time. But,  people mostly leave home in the morning, and the parameter $\gamma$ must be larger in the morning than the rest of the day. Similarly, since the people return home late in the afternoon, the parameter $\chi$ must be larger during that period than during the rest day. For each $y \in \R^2$, therefore, the return-to-home model with circadian rhythm (one-day periodic parameters) reads as follows
\begin{equation}\label{7.2}
	\left\{\begin{array}{ll}
		\partial_{t} u_y(t)= \chi(t) \int_{\R^2}w_y(t,x)\dx-\gamma(t) u_y(t),   \vspace{0.2cm} \\
		\partial_{t} v_y(t,x)=\varepsilon^2\Delta_{x}v_y- \mathbf{\nabla}_x \cdot  \left( v_y\, \mathbf{C_y}	 \right)-\alpha v_y+\gamma(t) g(x-y)u_y(t),    \vspace{0.2cm} \\
		\partial_{t} w_y(t,x)=\alpha v_y(t,x)-\chi(t) w_y(t,x),
	\end{array}\right.
\end{equation}
where the function $t \to \gamma(t)$ and $t \to \chi(t)$ are one-day periodic functions. 

\medskip 	
To conclude, we should insist on the fact that in the model, the individuals return home instantaneously. So here, we use diffusion and convection processes to derive the distribution of individuals at work from the distribution of individuals at home. In most practical problems, such as epidemic outbreaks and others, the two distributions will be sufficient to understand the major interactions between individuals. 

\newpage 

\appendix
\begin{center}
	{\LARGE	\textbf{Appendix}}
\end{center}
\section{The return home model on a bounded domain}
\label{SectionA}
We consider the rectangle domain of $\R^2$
$$
\Omega=\left(a_1,b_1 \right) \times \left(a_2,b_2 \right) = \{\left(x_1,x_2 \right) \in \R^2 : a_{1}<x_1<b_{1}, \text{ and } a_{2}<x_2<b_{2}\}.
$$ 
The return home model with no flux  at the boundary (i.e. with Neumann boundary conditions) is the following
\begin{equation}\label{A.1}
	\left\{\begin{array}{l}
		\partial_{t} u(t, y)=\chi \int_{\Omega}w(t,x,y)\dx-\gamma u(t,y),\vspace{0.2cm}  \\
		\partial_{t} v(t,x,y)=\varepsilon^2\Delta_{x}v(t,x,y)-\alpha v(t,x,y)+\gamma \rho(x,y)u(t,y),\vspace{0.2cm}  \\
		\partial_{t} w(t,x,y)=\alpha v(t,x,y)-\chi w(t,y), \vspace{0.2cm} \\
		
	\end{array}\right.
\end{equation}
with $ t\geq 0, x\in\Omega, y\in\Omega$, and in order to preserve the $\Li^1$ norm in space, we impose Neumann boundary conditions. As $\Omega$ is assumed to be a rectangle, that is 
\begin{equation}\label{A.2}
	\left\{\begin{array}{l}
		\partial_{x_1}v(t,x,y)= 0, t\geq0, x_1 =a_1 \text{ or } x_1 =b_1, \vspace{0.2cm} \\
		\partial_{x_2}v(t,x,y)= 0, t\geq0,  x_2 =a_2 \text{ or } x_2 =b_2,
	\end{array}\right.
\end{equation}
the initial distribution at $t=0$ 
\begin{equation} \label{A.3}
	\left\{	\begin{array}{l}
		u(0,y)=u_0(y)\in L_{+}^{1}(\Omega,\R) , \vspace{0.2cm}  \\ 
		v(0,x,y)=v_0(x,y)\in L_{+}^{1}(\Omega\times\Omega,\R), \vspace{0.2cm}  \\
		\text{ and }  \vspace{0.2cm}  \\
		w(0,x,y)=w_0(x,y)\in L_{+}^{1}(\Omega\times\Omega,\R).
	\end{array}\right.
\end{equation}
In order to preserve the total number of individuals, we  defined for  $x=\left(x_1, x_2 \right) \in\Omega$, and $y=\left(y_1, y_2 \right)  \in\Omega$, as follows
$$
\rho(x,y)=\dfrac{g(x-y)}{G(y)},
$$
where $G(y)$ is a normalization constant, which is defined by 
$$
G(y)=\int_{\Omega}g(x-y)\dx, \forall y \in \Omega.
$$ 
\begin{remark} In the formula for $\rho(x,y)$ we divide $g(x-y)$ by $G(y)$, in order to obtain 
	$$
	\int_{\Omega} \rho(x,y)dx=1, \forall y\in \Omega. 
	$$
\end{remark}
In Figure \ref{Fig1} we plot the function $(y_1,y_2) \to G(y_1,y_2)$ and we use the $2$ dimensional Simpson method to compute the integrals. 
\begin{figure}[H]
	\begin{center}
		\includegraphics[scale=0.23]{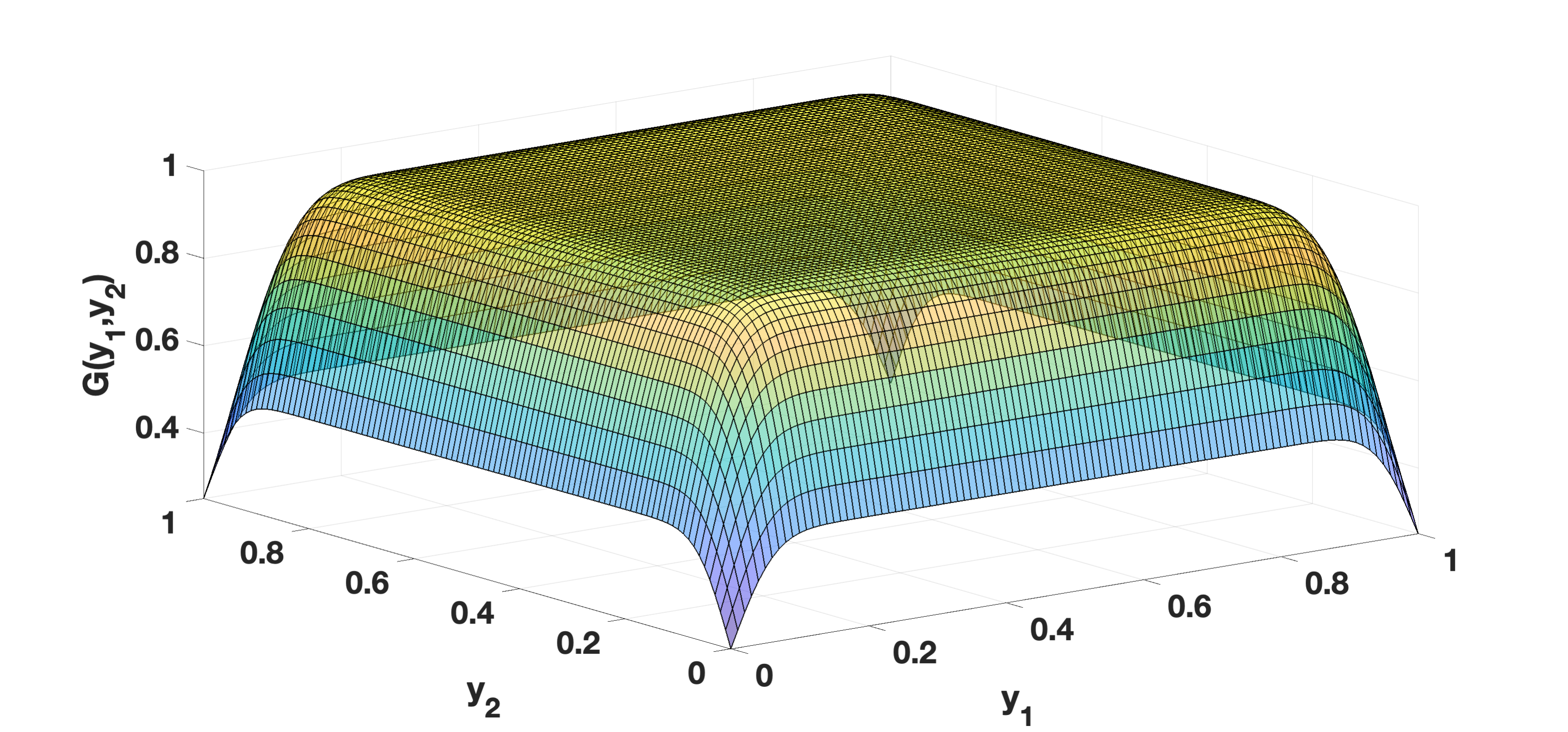}
	\end{center}
	\caption{\textit{In this figure we plot $(y_1,y_2) \to G(y_1,y_2)$. Here we use $\Omega=[0,1]\times [0,1]$ and the Gaussian function $g(x_1,x_2)$ with $\sigma=0.05$. }}\label{Fig0}
\end{figure}

\section{Matrix form of the numerical scheme}
From Appendix A, we know that the unknowns and equations are stored ``naturally'' as components of a vector for the one-dimensional case. However, for the two-dimensional case, we need to deal directly with the components of a matrix. Rearranging the values as a column vector raises the delicate issue of grid point renumbering. 

We define for each  $ i= 1,\cdots,n_1$, $ j= 1,\cdots,n_2$, $k= 1,\cdots,n_1$, and $l= 1,\cdots,n_2,$
and set
$$
m_1= (j - 1)n_1 + i  \in [1, n_1 n_2]\Leftrightarrow i= {\rm mod}(m_1, n_1), \text{ and } j=\dfrac{m_1-i}{n_1}+1, 
$$
and
$$
m_2= (l - 1)n_1 + k \in [1, n_1 n_2] \Leftrightarrow k= {\rm mod}(m_2, n_1), \text{ and } l=\dfrac{m_2-k}{n_1}+1, 
$$
and 
$$
m=(m_1-1) (n_1 n_2)+m_2   \in \left[ 1, \left( n_1 n_2\right)^2 \right] \Leftrightarrow m_2= {\rm mod}(m, n_1 n_2), \text{ and } m_1=\dfrac{m-m_2}{n_1 n_2}+1.
$$
We agree to note the grid points from ``the left to the right'' and from ``the bottom to the top'', i.e.,
according to the increasing order of the $i$, $j$ and $k$, $l$ indices, respectively.  Hence, $m_1$ and $m_2$ are the numbers corresponding to the points $(x_{1i}, x_{2j})$ and $(y_{1k}, y_{2l})$, respectively.

The vector $v$ is then defined by its components 
\begin{equation*}\label{A.1}
	v(m_1)^{n} = v(t^n, x_{1i}, x_{2j}), \quad \forall i= 1,\cdots,n_1, \forall  j= 1,\cdots,n_2. 
\end{equation*}
It follows from Appendix A that the discrete problem can be written in the vector form as follows:
\begin{equation*}\label{A.2}
	\Delta_{x}v(t^n, x_{1i}, x_{2j})=Av(m_1)^n,
\end{equation*}
where $A \in M_{n_1\times n_2}\left( \R \right)$ is the block tridiagonal matrix defined as
\begin{equation*}\label{A.3}
	A=
	\begin{pNiceMatrix}
		\dfrac{B}{\Delta x_1^2}  -\dfrac{I}{\Delta x_2^2}  & \dfrac{I}{\Delta x_2^2}       &0        & \Cdots       &0\\
		\dfrac{I}{\Delta x_2^2}   & 	\dfrac{B}{\Delta x_1^2}  -\dfrac{2 I}{\Delta x_2^2} 	    & \dfrac{I}{\Delta x_2^2}     &  \Ddots       &\Vdots \\
		
		0&\Ddots  &\Ddots  &  \Ddots      &0\\
		\Vdots& \Ddots      &\dfrac{I}{\Delta x_2^2}   & \dfrac{B}{\Delta x_1^2}  -\dfrac{2 I}{\Delta x_2^2} 	  &  \dfrac{I}{\Delta x_2^2}   \\
		0& \Cdots       &  0      & \dfrac{ I}{\Delta x_2^2}   & 	\dfrac{B}{\Delta x_1^2}  -\dfrac{I}{\Delta x_2^2}  
	\end{pNiceMatrix},
\end{equation*}
where $0$ is a $n_1\times n_1$ null matrix and $I$ denotes the $n_1\times n_1$ identity matrix, and 
\begin{equation*}\label{A.4}
	B	=
	\begin{pNiceMatrix}
		-1& 1      &  0      &    & \Cdots        &0\\
		1  & -2     &      & \Ddots   &         &\Vdots\\
		0  &      &      &        &         & \\
		&\Ddots  &\Ddots   &\Ddots    &\Ddots     &0\\
		\Vdots &        &         &       & -2      & 1 \\
		0  & \Cdots &         & 0        & 1       & -1
	\end{pNiceMatrix}  \in M_{n_1} \left( \R \right).
\end{equation*}
The matrix $A$ rewrite as 
\begin{equation*}\label{A.5}
	A=
	\begin{pNiceMatrix}
		\dfrac{B}{\Delta x_1^2} +\dfrac{I}{\Delta x_2^2}  & \dfrac{I}{\Delta x_2^2}       &0        &\Cdots        &0\\
		\dfrac{I}{\Delta x_2^2}   & 	\dfrac{B}{\Delta x_1^2}    &  \dfrac{I}{\Delta x_2^2}     &\Ddots        &\Vdots\\
		0&\Ddots  &\Ddots  &\Ddots  &0\\
		\Vdots& \Ddots       &\dfrac{I}{\Delta x_2^2}   & 	\dfrac{B}{\Delta x_1^2}     &  \dfrac{I}{\Delta x_2^2}   \\
		0&   \Cdots     &    0    & \dfrac{2 I}{\Delta x_2^2}   & 	\dfrac{B}{\Delta x_1^2}  +\dfrac{I}{\Delta x_2^2}
	\end{pNiceMatrix}-\dfrac{2I}{\Delta x_2^2} .
\end{equation*}
Therefore, we deduce that for $m_2$ fixed, we obtain a system of equations when  $m_1$ varies. Define 
$$
V_{m_2}^n=\left[ v(m_1)^{n}_{m_2} \right]_{m_1=1}^{m_1=n_1\times n_2},
$$ 
$$
R_{m_2}=\left[ \rho\left(x_{m_1},y_{m_2} \right)  \right]_{m_1=1}^{m_1=n_1\times n_2}.
$$
Then system \eqref{A.1} can be written as a semi-implicit numerical scheme 
\begin{equation}\label{A.9}
	\left\{
	\begin{array}{rl}
		u_{m_2}^{n+1}=& u_{m_2}^{n}+\Delta t \Delta x_1 \Delta x_2 \displaystyle \chi \,  \sum W_{m_2}^{n}-\Delta t  \,  \gamma \,  u_{m_2}^n, \vspace{0.2cm}\\
		V_{m_2}^{n+1}= &V_{m_2}^{n}+\Delta t  \, \varepsilon		A V_{m_2}^{n+1}-\Delta t  \,  \alpha \, V_{m_2}^{n}+\Delta t  \, \gamma  \,  {\rm diag }\left(R_{m_2} \right) u_{m_2}^n, \vspace{0.2cm}\\
		W_{m_2}^{n+1}=& W_{m_2}^{n}+\Delta t  \,  \alpha \, V_{m_2}^{n}-\Delta t  \, \chi \, W_{m_2}^n.
	\end{array}
 \right.
\end{equation}
The complete problem with convection is more challenging to simulate. Nevertheless, it is possible to use splitting methods in that case. We refer to Speth, Green, MacNamara, and  Strang \cite{Speth} for more on this topic.

\end{document}